\documentclass[12pt]{amsart}

\usepackage{fullpage}
\usepackage{amssymb}
\usepackage{amsmath}
\usepackage{amsthm}
\usepackage{mathtools}
\usepackage[shortlabels]{enumitem}
\usepackage{colonequals}
\usepackage{xspace}
\usepackage{booktabs}
\usepackage{array}
\usepackage{colortbl}
\usepackage{caption}
\usepackage{tikz}
\usetikzlibrary{cd}
\usepackage[alphabetic,backrefs,lite,msc-links]{amsrefs}
\usepackage{color}
\usepackage{setspace}
\usepackage{hyperref}
\definecolor{mycitecolor}{RGB}{169,169,169}
\definecolor{mylinkcolor}{rgb}{1.0, 0.0, 0.5}
\definecolor{myurlcolor}{rgb}{1.0, 0.0, 0.5}
\hypersetup{colorlinks=true,
  urlcolor=myurlcolor,
  citecolor=mycitecolor,
  linkcolor=mylinkcolor,
  linktoc=page,
  breaklinks=true}
\usepackage{url}
\usepackage{cleveref}
\usepackage{indentfirst}

\newcommand{\A}{\mathbb{A}}
\newcommand{\F}{{\mathbb F}}
\renewcommand{\P}{\mathbb{P}}
\newcommand{\PP}{{\mathbb P}}
\newcommand{\Q}{\mathbb{Q}}
\newcommand{\Z}{\mathbb{Z}}
\newcommand{\Qbar}{{\overline{\Q}}}
\newcommand{\Pone}{\P^{1}}
\newcommand{\calF}{{\mathcal F}}
\newcommand{\calS}{{\mathcal S}}
\newcommand{\calU}{{\mathcal U}}
\newcommand{\mcU}{\calU}
\newcommand{\calV}{{\mathcal V}}
\newcommand{\calZ}{{\mathcal Z}}
\newcommand{\OO}{{\mathcal O}}

\newcommand{\OKS}{\OO_{K,\calS}}
\newcommand{\mfp}{{\mathfrak p}}

\newcommand{\x}{\mathsf{x}}
\newcommand{\y}{\mathsf{y}}
\newcommand{\w}{\mathsf{w}}
\renewcommand{\t}{\mathsf{t}}
\newcommand{\sfa}{\mathsf{a}}
\newcommand{\sfu}{\mathsf{u}}
\newcommand{\sfv}{\mathsf{v}}

\DeclareMathOperator{\ns}{ns}
\DeclareMathOperator{\spc}{sp}
\DeclareMathOperator{\Aut}{Aut}
\DeclareMathOperator{\Autsch}{\mathbf{Aut}} 
\DeclareMathOperator{\Br}{Br}
\DeclareMathOperator{\disc}{disc}
\DeclareMathOperator{\Gal}{Gal}
\DeclareMathOperator{\Norm}{Norm}
\DeclareMathOperator{\Proj}{Proj}
\DeclareMathOperator{\Res}{Res}
\DeclareMathOperator{\Sym}{Sym}
\DeclareMathOperator{\Tr}{Tr}
\newcommand{\Frob}{\operatorname{Frob}}
\newcommand{\GL}{\operatorname{GL}}
\newcommand{\PGL}{\operatorname{PGL}}
\newcommand{\SL}{\operatorname{SL}}

\newcommand{\cdef}[1]{{\color{blue} \textsf{#1}}} 
\newcommand{\Magma}{\texttt{Magma}\xspace}
\newcommand{\mclabel}[1]{\href{https://beta.lmfdb.org/ModularCurve/Q/#1}{\texttt{#1}}}

\definecolor{tabrule}{gray}{0.80}
\newcommand{\rowsep}{\arrayrulecolor{tabrule}\midrule[\lightrulewidth]%
  \arrayrulecolor{black}}
\newcommand{\headrule}{\midrule[\heavyrulewidth]}
\newcommand{\tablesetup}{\renewcommand{\arraystretch}{1}%
  \setlength{\extrarowheight}{0pt}%
  \setbox\strutbox\hbox{\vrule
    height \dimexpr 0.9\baselineskip + \fontdimen22\textfont2\relax
    depth \dimexpr 0.9\baselineskip - \fontdimen22\textfont2\relax
    width 0pt}%
  \setlength{\heavyrulewidth}{1.1pt}%
  \setlength{\lightrulewidth}{0.4pt}%
  \setlength{\aboverulesep}{0pt}%
  \setlength{\belowrulesep}{0pt}}
\newcolumntype{:}{!{\color{tabrule}\vrule width \lightrulewidth}}

\newdimen\tabrowpad
\newcommand{\padrows}[1]{\noalign{\global\tabrowpad=#1\relax}}
\newcommand{\heavyvrule}{%
  \vrule width \heavyrulewidth
  \hskip -\heavyrulewidth
  \smash{\vrule width \heavyrulewidth
    height \dimexpr \ht\strutbox + \tabrowpad + \heavyrulewidth \relax
    depth 0pt}%
  \vrule width 0pt height \dimexpr \ht\strutbox + \tabrowpad \relax
    depth \dimexpr \dp\strutbox + \tabrowpad \relax}
\newcolumntype{I}{!{\heavyvrule}}
\newcommand{\tabmatrix}[1]{\text{\normalsize
    $\begin{pmatrix*}[r] #1\end{pmatrix*}$}}
\definecolor{tabhead}{gray}{0.92}
\newcommand{\headrow}{\rowcolor{tabhead}}

\crefname{lemma}{Lemma}{Lemmas}
\crefname{proposition}{Proposition}{Propositions}
\crefname{corollary}{Corollary}{Corollaries}
\crefname{remark}{Remark}{Remarks}
\crefname{example}{Example}{Examples}
\crefname{definition}{Definition}{Definitions}
\crefname{theorem}{Theorem}{Theorems}
\crefname{section}{Section}{Sections}
\crefname{table}{Table}{Tables}
\crefname{figure}{Figure}{Figures}
\crefformat{equation}{(#2#1#3)}
\Crefformat{equation}{Equation~(#2#1#3)}

\usepackage{etoolbox}
\AtBeginEnvironment{lemma}{\crefalias{theorem}{lemma}}
\AtBeginEnvironment{proposition}{\crefalias{theorem}{proposition}}
\AtBeginEnvironment{corollary}{\crefalias{theorem}{corollary}}
\AtBeginEnvironment{remark}{\crefalias{theorem}{remark}}
\AtBeginEnvironment{example}{\crefalias{theorem}{example}}
\AtBeginEnvironment{definition}{\crefalias{theorem}{definition}}
\AtBeginEnvironment{maincor}{\crefalias{mainthm}{maincor}}

\numberwithin{equation}{section}
\newtheorem{theorem}{Theorem}[section]
\newtheorem{lemma}[theorem]{Lemma}
\newtheorem{corollary}[theorem]{Corollary}
\newtheorem{proposition}[theorem]{Proposition}
\newtheorem{mainthm}{Theorem}
\newtheorem{maincor}[mainthm]{Corollary}
\theoremstyle{definition}
\newtheorem{definition}[theorem]{Definition}
\newtheorem{example}[theorem]{Example}
\newtheorem{remark}[theorem]{Remark}
\crefname{mainthm}{Theorem}{Theorems}
\crefname{maincor}{Corollary}{Corollaries}

\date{\today}
\title{$7$-adic Galois representations of elliptic curves over the rationals
  via Kummer descent}

\author{Santiago Arango-Pi{\~n}eros}
\address{Department of Mathematics, University of Massachusetts Amherst, Amherst, MA 01003, USA}
\email{santiago.arango.pineros@gmail.com}
\urladdr{\url{https://sarangop1728.github.io/}}

\author{David Zureick-Brown}
\address{Department of Mathematics, Amherst College, 31 Quadrangle Dr., Amherst, MA 01002, USA}
\email{dzureickbrown@amherst.edu}
\urladdr{\url{https://dmzb.github.io/}}

\begin{document}

\begin{abstract}
  We show that the two modular curves of level $49$ and genus $9$ left open by
  the work of Rouse, Sutherland and Zureick-Brown, and of Furio and Lombardo,
  have no non-CM rational points. Together with the theorem of Furio and
  Lombardo on $X_{\ns}^{+}(49)$, this completes the classification of the
  $7$-adic images of Galois of non-CM elliptic curves over $\Q$. The proof is a
  Kummer descent on the superelliptic equations $F(\x,\y) = k\,\w^{7}$ of Furio
  and Lombardo, in the cases $7 \mid k$ that they left open. The covering
  curves are twists of the Fermat septic, they map to twists of the Klein
  quartic, and a $7$-adic computation shows that the twist attached to a
  solution is trivial. The rational points of the Klein quartic were
  determined by Hurwitz, who reduced the question to Fermat's Last Theorem for
  exponent $7$, proved by Lam\'e. Thus, the last open case of the
  $7$-adic part of Mazur's Program B rests on Fermat's Last Theorem
  for exponent $7$.
\end{abstract}

\maketitle

\setcounter{tocdepth}{1}
\tableofcontents

\section{Introduction}
\label{sec:introduction}

Let $E$ be an elliptic curve over $\Q$ without complex multiplication and let
$\ell$ be a prime. By Serre's open image theorem \cite{Serre-72}, the
$\ell$-adic Galois representation $\rho_{E,\ell^{\infty}}$ of $\Gal_{\Q}$ has open image in
$\GL_2(\Z_\ell)$, and Mazur's \emph{Program B} \cite{Mazur-77}*{p.~109} asks for
the classification of the subgroups that occur. Rouse, Sutherland and the
second author \cite{RSZB-22}*{Theorem 1.6} carried out this classification up
to the determination of the rational points of a short list of modular curves
(we give their LMFDB labels \cite{LMFDB}): the curves $X_{\ns}^{+}(\ell)$
attached to the normalizer of a non-split Cartan subgroup, for primes
$\ell > 17$; the four curves $X_{\ns}^{+}(27) =$ \mclabel{27.243.12.a.1},
$X_{\ns}^{+}(25) =$ \mclabel{25.250.14.a.1}, $X_{\ns}^{+}(49) =$
\mclabel{49.1029.69.a.1} and $X_{\ns}^{+}(121) =$ \mclabel{121.6655.511.a.1};
and two curves of level $49$ and genus $9$, which, following Furio and Lombardo
\cite{Furio--Lombardo26}*{Definition 2.7, Remark 2.8}, we denote by
$X_{\ns}^{\sharp}(49) =$ \mclabel{49.147.9.a.1} and
$X_{\spc}^{\sharp}(49) =$ \mclabel{49.196.9.a.1}. The curve $X_{\ns}^{+}(27)$
has since been settled by Balakrishnan, Betts, Hast, Jha and M\"uller
\cite{BBHJM-25}, and $X_{\ns}^{+}(49)$ by Furio and Lombardo, who proved that it
has exactly seven rational points, all of them CM
\cite{Furio--Lombardo26}*{Theorem 1.4}. Our work here builds on that of Furio
and Lombardo, and our main result is the computation of the rational points on
the curves $X_{\ns}^{\sharp}(49)$ and $X_{\spc}^{\sharp}(49)$, thereby
completing the classification of $7$-adic images of Galois for non-CM elliptic
curves over $\Q$.

Furio and Lombardo reduced the computation of the rational points on
$X_{\ns}^{\sharp}(49)$ and $X_{\spc}^{\sharp}(49)$ to the determination of the rational
points of a single plane quartic $X_{E_3}(7)$, a twist of the Klein quartic
whose Jacobian has Mordell--Weil rank $3$, and they conjectured that this
quartic has exactly four rational points \cite{Furio--Lombardo26}*{Conjecture
  1.6}. We do not address this conjecture: our proof is unconditional and takes
a different route.

\begin{remark}
  \label{remark:eray}
  On September 18, 2026, Eray Karabiyik communicated to us that he was able
  to prove
  \cite{Furio--Lombardo26}*{Conjecture
    1.6} directly \cite{Karabiyik-26}. Karabiyik's proof, in essence, does
  quadratic Chabauty over number fields by using an elliptic quotient over a
  number field of degree 21.
\end{remark}

\begin{remark}
  On September 28, 2026, Nguyen Xuan Tho emailed us a preprint
  \cite{Xuan2026} with an AI-generated proof of
  \cite{Furio--Lombardo26}*{Conjecture
    1.6}.
\end{remark}

\subsection*{The Diophantine statement}
The curves $X_{\ns}^{\sharp}(49)$ and $X_{\spc}^{\sharp}(49)$ admit natural maps
of degree $7$ to the genus zero curves $X_{\ns}^{+}(7) =$ \mclabel{7.21.0.a.1} and
$X_{\spc}^{+}(7) =$ \mclabel{7.28.0.a.1} attached to the normalizers of a
non-split and of a split Cartan subgroup of $\GL_2(\F_7)$ (the precise
definitions are recalled in \cref{sec:groups}). The $j$-maps of
$X_{\ns}^{+}(7)$ and $X_{\spc}^{+}(7)$, in Zywina's coordinates (the functions
$J_6$ and $J_2$ of \cite{Zywina-15}*{Section 1.4}; see
\cite{Zywina-15}*{Theorem 1.5(ii)} and \cite{Furio--Lombardo26}*{Theorem 2.1}),
have denominators $f_{\ns}(\t)^7$ and $f_{\spc}(\t)^7$, where
\begin{equation}
  \label{eq:the-cubics}
  f_{\ns}(\t) \colonequals \t^3 - 7\t^2 + 7\t + 7,
  \qquad
  f_{\spc}(\t) \colonequals \t^3 - 4\t^2 + 3\t + 1 .
\end{equation}
Both cubics define the field $K \colonequals \Q(\zeta_7)^{+}$. Let $F(\x,\y)$
denote the homogenization of $f(\t)$:
\begin{equation}
  \label{eq:the-cubic-forms}
  F_{\ns}(\x,\y) \colonequals \x^3 - 7\x^2\y + 7\x\y^2 + 7\y^3,
  \qquad
  F_{\spc}(\x,\y) \colonequals \x^3 - 4\x^2\y + 3\x\y^2 + \y^3 .
\end{equation}
Furio and Lombardo showed
\cite{Furio--Lombardo26}*{Proposition 3.1} that a non-CM rational point on
$X_{\ns}^{\sharp}(49)$ or $X_{\spc}^{\sharp}(49)$ yields coprime integers $x,y$
and an integer $w$ such that $F(x,y) = k w^7$, with $k \in \{7,56\}$ in the
non-split case and $k \in \{1,7\}$ in the split case; we rederive this in
\cref{sec:equations}. Our main result is the solution of these equations when
$7 \mid k$.

\begin{mainthm}
  \label{thm:superelliptic}
  Let $(F,k)$ be one of the pairs $(F_{\ns},7)$, $(F_{\ns},56)$,
  $(F_{\spc},7)$. The integer solutions $(x,y,w)$ of
  \begin{equation}
    \label{eq:superelliptic-intro}
    F(\x,\y) = k\,\w^{7}, \qquad \gcd(x,y) = 1,
  \end{equation}
  all satisfy $w = \pm 1$. Moreover, they are exactly the ones listed in
  \cref{tab:solutions}.
\end{mainthm}

\begin{table}[ht]
  \centering
  \tablesetup
  \begin{tabular}{IcIc:cc:cccI}
    \toprule
    \headrow
    $(F,k)$ & $(F_{\ns},7)$ & \multicolumn{2}{c:}{$(F_{\ns},56)$}
            & \multicolumn{3}{cI}{$(F_{\spc},7)$} \\
    \headrule
    $(x,y,w)$ & $\pm (0,1,1)$ & $\pm (7,1,1)$ & $\pm (7,3,-1)$
              & $\pm (1,-1,1)$ & $\pm (5,2,-1)$ & $\pm (4,3,1)$ \\
    \headrule
    \headrow
    curve & \multicolumn{3}{c:}{$X_{\ns}^{+}(7)(\Q)$}
          & \multicolumn{3}{cI}{$X_{\spc}^{+}(7)(\Q)$} \\
    \headrule
    $t = x/y$ & $0$ & $7$ & $7/3$ & $-1$ & $5/2$ & $4/3$ \\
    \rowsep
    $j(t)$ & $0$ & non-CM & non-CM & $0$ & non-CM & non-CM \\
    \bottomrule
  \end{tabular}
  \caption{The solutions of \eqref{eq:superelliptic-intro} for the three pairs
    $(F,k)$ with $7 \mid k$ (\cref{thm:superelliptic}), the corresponding
    rational points $t = x/y$ of $X_{\ns}^{+}(7)$ and $X_{\spc}^{+}(7)$, and
    whether their $j$-invariants are CM.}
  \label{tab:solutions}
\end{table}

For the remaining pairs, those with $7 \nmid k$, the analogous statement is a
theorem of Furio and Lombardo: every solution of $F_{\ns}(x,y) = kw^7$ with
$k \in \{1,8\}$ and of $F_{\spc}(x,y) = w^7$ in coprime integers has
$w = \pm 1$ \cite{Furio--Lombardo26}*{Corollary 3.3(1), Theorem 3.5(1)}. We use this in the
split case of the next theorem, where $k = 1$ is allowed. The cases $7 \mid k$ are
precisely the ones that lead Furio and Lombardo to the twist of the Klein
quartic of rank $3$.

\begin{mainthm}
  \label{thm:modular-curves}
  Every rational point of $X_{\ns}^{\sharp}(49)$ and of $X_{\spc}^{\sharp}(49)$
  is a cusp or a CM point. More precisely, every non-cuspidal rational point of
  either curve has $j$-invariant $0$. In particular, the four non-CM points of
  \cref{tab:solutions} do not lift to rational points of
  $X_{\ns}^{\sharp}(49)$, respectively of $X_{\spc}^{\sharp}(49)$.
\end{mainthm}

\begin{maincor}
  \label{cor:seven-adic-images}
  Theorem 1.7 of \cite{Furio--Lombardo26} holds unconditionally. In particular, for every
  elliptic curve $E$ over $\Q$ without complex multiplication, the $7$-adic
  image $G = \rho_{E,7^{\infty}}(\Gal_{\Q})$ satisfies one of the first two
  alternatives of \cite{RSZB-22}*{Theorem 1.6}: either $X_G$ has infinitely many
  rational points and $G$ is one of the groups of Sutherland and Zywina
  \cite{SZ-17}, or $G$ appears in \cite{RSZB-22}*{Table 1}.
\end{maincor}
\begin{proof}
  Alternative (iii) of \cite{RSZB-22}*{Theorem 1.6} at level $7^2$ is excluded by
  \cite{Furio--Lombardo26}*{Theorem 1.4}, and alternative (iv) by \cref{thm:modular-curves}.
\end{proof}

\subsection*{Summary of the proofs}
Furio and Lombardo attach to a solution of $F(x,y) = kw^7$ a solution of a
generalized Fermat equation of signature $(2,3,7)$, by means of the covariants
of the cubic form $F$ \cite{Furio--Lombardo26}*{Proposition 3.2, Corollary 3.3} (compare
Bennett and Dahmen \cite{BD-13}*{Equation (4)}), and then apply the modular
method of Poonen, Schaefer and Stoll \cite{PSS-07}*{Section 4}: the solutions
are governed by rational points on twists $X_E(7)$ of the Klein quartic
\cite{HK-03}. When $7 \nmid k$ the equation is $a^2 + 28b^3 = 27c^7$ and all the
twists that occur have rank at most $2$; when $7 \mid k$ it is
$a^2 + 196b^3 = 27c^7$, and one twist of rank $3$ remains.

We instead descend directly on the superelliptic equation, in the style of
Kummer \cite{DG-95}*{Section 2.1}. This method is streamlined and explained by
Bruin in \cite{Bruin02-thesis}*{Section 3.1}, \cite{Bruin06-nn2}*{Section 3}.
The argument has six steps. Along the way, the set of classes by which we have
to twist shrinks: from the Selmer group $K(7,\calS)$ to a subset $\Delta$ of
$49$ classes in step~(2), and to $7$ of these in step~(4).

\begin{enumerate}[leftmargin=*, itemsep=0.4em]
\item \textbf{The equations (\cref{sec:equations}).} A rational point of
  $X_{\ns}^{\sharp}(49)$ or $X_{\spc}^{\sharp}(49)$ lies over a point $t = x/y$
  of the genus zero curve $X_{\ns}^{+}(7)$, respectively $X_{\spc}^{+}(7)$, and the
  exponent $7$ in $F(x,y) = kw^7$ comes from the primes of potentially
  multiplicative reduction. At such a prime the Tate curve makes inertia act
  unipotently, and the shape of the level $49$ group forces the valuation of
  $F(x,y)$ to be divisible by $7$. This recovers
  \cite{Furio--Lombardo26}*{Proposition 3.1}. In the split case the argument of
  \cite{Furio--Lombardo26} uses an element of the decomposition group that does not exist
  at the primes $p$ with $p^{6} \equiv 1 \pmod{49}$; the proof given here uses
  only inertia, together with the position of the rational cusp
  (\cref{lem:cusp-reduction}, \cref{rem:FL-split}).
\item \textbf{Kummer descent to twists of the Fermat septic
    (\cref{sec:descent}).} Let $\theta \in K$ be a root of $f$. Over $K$ the form
  $F(\x,\y)$ is the norm of $\x - \theta \y$, and since $K$ has class number one, a
  solution $(x,y,w)$ to \cref{eq:superelliptic-intro} yields
  $x - \theta y = \delta\omega^7$ with $\omega \in K^{\times}$ and $\delta$ an
  $\calS$-unit, i.e.,~a unit outside the set $\calS$ of primes dividing
  $k\disc(f)$. So $x-\theta y$ is a
  seventh power in $K$ up to a factor $\delta$ that only matters through its
  class in the Selmer group $K(7,\calS) = \OKS^{\times}/\OKS^{\times 7}$, and
  comparing norms restricts this class to an explicit subset $\Delta(\theta,k)$ of
  $49$ classes (\cref{prop:bruin}). Each $\delta \in \Delta(\theta,k)$ gives a plane
  curve, as follows. Write $\omega = a_0 + a_1\theta + a_2\theta^2$ with $a_n \in
  \Q$, and put $\mathbf{a} = (a_0,a_1,a_2)$. There are homogeneous ternary
  forms $P^{(0)}_{\theta,\delta}, P^{(1)}_{\theta,\delta},
  P^{(2)}_{\theta,\delta}$ of degree $7$ with rational coefficients such that
  \[
    \delta\,(a_0 + a_1\theta + a_2\theta^2)^{7}
    = P^{(0)}_{\theta,\delta}(\mathbf{a}) + P^{(1)}_{\theta,\delta}(\mathbf{a})\,\theta + P^{(2)}_{\theta,\delta}(\mathbf{a})\,\theta^{2}.
  \]
  Since $x - \theta y$ has no $\theta^{2}$ term, the point $(a_0:a_1:a_2)$ lies on the plane septic
  $C_{\theta,\delta} \colon P^{(2)}_{\theta,\delta} = 0$ over $\Q$, of genus
  $15$, and $(x : y) = (P^{(0)}_{\theta,\delta}(\mathbf{a}) : -P^{(1)}_{\theta,\delta}(\mathbf{a}))$.
  So the rational points of the curves $C_{\theta,\delta}$ account for all the
  solutions. Over $K$ each of these curves becomes a diagonal
  septic, a twist of the Fermat curve $\calF\colon \sfu_0^7 + \sfu_1^7 + \sfu_2^7
  = 0$ to which it is isomorphic over $\Qbar$. Remarkably, all six pairs $(F,k)$
  that arise in \cref{sec:equations} lead to the same $49$ twists of the Fermat
  septic.
\item \textbf{Quotient to twists of the Klein quartic (\cref{sec:klein}).} The
  Fermat septic $\calF$ is a cyclic unramified cover of degree $7$ of the
  Klein quartic curve
  $\calZ \colon \sfv_0^3\sfv_1 + \sfv_1^3\sfv_2 + \sfv_2^3\sfv_0= 0$
  \cite{Elkies-99}*{Section 3.2}. This cover survives the twisting: each of our
  septics $C_{\theta,\delta}$ maps, over $\Q$, onto a twist $Z_{\theta,\delta}$ of the Klein
  quartic, so it is enough to find the rational points of these plane quartics.
  Which twist occurs is recorded by the class of $\delta/f'(\theta)$ in the
  group $K^{\times}/\Q^{\times}(K^{\times})^{3+\sigma}$
  (\cref{lem:twisting}). For $\delta \in \Delta(\theta,k)$ this class lies in a
  subgroup of order $7$, so the $49$ septics of step~(2) map to at most $7$
  quartics up to isomorphism, and in fact to only three
  (\cref{prop:klein-twists}).
\item \textbf{The $7$-adic step (\cref{sec:proof-thm-1}).} When
  $7 \mid F(x,y)$, a short computation in the completion of $K$ at $7$ shows that
  this class vanishes (\cref{prop:seven-adic}): a homomorphism $\kappa$ to
  $\F_7$ must vanish on it. This rules out $42$ of the $49$ classes
  $\delta \in \Delta(\theta,k)$, and the remaining $7$ all give quartics $\Q$-isomorphic
  to one fixed quartic $Z_0$. A classical change of coordinates of Klein,
  defined over $K$ \cite{Klein-1879}*{\S 5}, \cite{Elkies-99}*{(1.3)},
  gives an isomorphism over $\Q$ between $Z_0$ and the Klein quartic itself
  (\cref{prop:klein-identity}, \cref{cor:Z0}). Conversely, for the other $42$
  classes the quartic $Z_{\theta,\delta}$ is not isomorphic to the Klein quartic
  over $\Q$ (\cref{prop:klein-twists}).
  The cases $7 \mid k$ are exactly those that were missing from \cite{Furio--Lombardo26}.
\item \textbf{Hurwitz (\cref{lem:hurwitz}).} The Klein quartic has only its
  three obvious rational points: Hurwitz \cite{Hurwitz-08} showed that any other
  would give a nontrivial solution of Fermat's equation of exponent $7$, and
  Elkies \cite{Elkies-99}*{Section 3.1} gives a direct proof through an
  elliptic quotient of rank zero.\footnote{Elkies also pointed out to us that an argument equivalent to Hurwitz's appears even earlier, in an 1864 paper of Genocchi \cite{Genocchi-64}.} Pulling the three points back gives the
  solutions listed in \cref{thm:superelliptic}. So, via Hurwitz, the last open
  case of the $7$-adic Program B rests on Fermat's Last Theorem for exponent
  $7$.
\item \textbf{From \cref{thm:superelliptic} to \cref{thm:modular-curves}
    (\cref{sec:proof-thm-2}).} \Cref{thm:superelliptic}, together with the
  result of Furio and Lombardo for $k = 1$ recalled before \cref{thm:modular-curves}, leaves eight non-CM
  $j$-invariants of points of $X_{\ns}^{+}(7)$ or $X_{\spc}^{+}(7)$ that could still lift to a rational point of
  $X_{\ns}^{\sharp}(49)$ or $X_{\spc}^{\sharp}(49)$ (\cref{tab:candidates}).
  One could rule this out using the equations for the maps $X_{\bullet}^{\sharp}(49) \to X_{\bullet}^{+}(7)$; instead we
  use the following group-theoretic criterion (\cref{cor:criterion}). Let $E$ be an
  elliptic curve over $\Q$ whose mod-$49$ Galois image is contained, up to
  conjugacy, in one of the groups $G_{\ns}^{\sharp}(49)$ or
  $G_{\spc}^{\sharp}(49)$ defining these two curves (\cref{def:G-sharp}), and let
  $p \neq 7$ be a prime of good reduction, with trace of Frobenius $a_p$. If
  $7 \nmid a_p(a_p^2 - 4p)$, then $\Frob_p^{48}$ acts on $E[49](\Qbar)$ as a
  scalar. Indeed, modulo $7$ such a Frobenius lies in the Cartan subgroup, so
  its $48$th power is $1$ modulo $7$, and the shape of the level $49$ group
  forces this power to be scalar modulo $49$. For each of the eight candidates,
  the criterion fails at a small prime $p$. The CM points are handled by the
  same criterion, together with \cite{Furio--Lombardo26}*{Lemma 2.11} in the non-split
  case.
\end{enumerate}

The relevance of quotients of Fermat curves to descent problems of this kind
goes back at least to Faddeev \cite{Faddeev-61}, Gross and Rohrlich
\cite{GR-78}, McCallum \cites{McCallum-88}, and McCallum and Tzermias
\cite{MT-03}. Twists of the Fermat septic covering twists of the Klein quartic
also appear in \cite{PSS-07}*{Section 10.1}. \Cref{fig:tower} summarizes the
curves and maps involved.

\begin{figure}[ht]
  \centering
  \begin{tikzcd}[column sep=large, row sep=large]
    \calF \arrow[r, "\psi"] \arrow[d, dotted, "\cong"'] & \calZ \arrow[d, dotted, "\cong"] & \\
    C_{\theta,\delta} \arrow[r, "\psi"] \arrow[dr, "\phi_{\theta,\delta}"'] &
    Z_{\theta,\delta} \arrow[d, "\varphi_{\theta,\delta}"] &
    X_{\bullet}^{\sharp}(49) \arrow[dl] \\
    & \Pone_{\t} = X_{\bullet}^{+}(7) \arrow[d, "j_{\bullet}"] & \\
    & X(1) &
  \end{tikzcd}
  \caption{The curves of the proof, for one $\delta \in \Delta(\theta,k)$. The two
    vertical isomorphisms (dotted) are defined over $\Qbar$, everything else over $\Q$. The maps $\psi$ and
    $\varphi_{\theta,\delta}$ have degree $7$, and $\phi_{\theta,\delta}$ has
    degree $49$ and is branched exactly over the three roots of $f$, which are
    cusps of $X_{\bullet}^{+}(7)$. The covering $X_{\bullet}^{\sharp}(49) \to
    X_{\bullet}^{+}(7)$ has degree $7$, and $j_{\bullet}$ has degree $21$ for
    $\bullet = \ns$ and $28$ for $\bullet = \spc$. A solution $(x,y,w)$ of
    $F(x,y) = kw^{7}$ is a rational point $(x:y)$ of $\Pone_{\t}$ in the image of
    $\phi_{\theta,\delta}$ for some $\delta \in \Delta(\theta,k)$. When $7 \mid k$, every
    solution comes from a $\delta$ with $\kappa(\delta/f'(\theta)) = 0$; these
    are exactly the $\delta$ for which $Z_{\theta,\delta}$ is isomorphic to
    $\calZ$ over $\Q$ (\cref{prop:seven-adic}, \cref{cor:Z0},
    \cref{prop:klein-twists}), and the three rational points of
    that quartic map to $t \in \{0, 7, 7/3\}$ for $\bullet = \ns$ and to
    $t \in \{-1, 5/2, 4/3\}$ for $\bullet = \spc$, the points of
    \cref{thm:superelliptic}.}
  \label{fig:tower}
\end{figure}
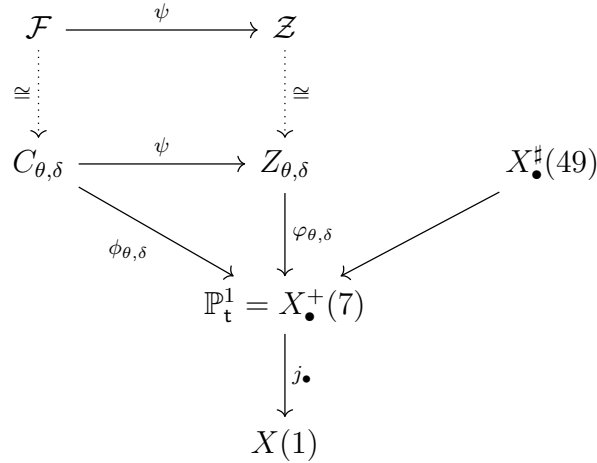

\subsection*{Code}
Every computational claim in this paper is verified by \Magma{} \cite{Magma}
code. The code and the source of this paper are available at
\url{https://github.com/sarangop1728/7-adic-settlers-of-cartan}. Both are generated
from one literate file, in which each statement of the paper is followed by the
code that verifies it. There is one
script for each of
\cref{sec:equations,sec:descent,sec:klein,sec:proof-thm-1,sec:proof-thm-2}, and
a script \texttt{descent.m} that runs the descent of
\cref{sec:descent,sec:klein,sec:proof-thm-1} from the Selmer sets to
\cref{tab:solutions}; each runs in under a minute. A web version of the
paper, at \url{https://sarangop1728.github.io/7-adic-settlers-of-cartan/}, shows the
code and its output below the statements it verifies.

\subsection*{AI use disclosure}
Throughout this project we used AI models (GPT-5.6 Sol, and Claude Opus 5, Opus
5.5, and Fable 5.1) for mathematical exploration, to find simplifications of
our arguments, to write and run \Magma{} code, and for copy-editing. The origin
of the main ideas and the precise mathematical use of AI are described in more
detail in \cref{sec:ai}. All mathematical results in this paper, and the code
and computations verifying them, have been thoroughly vetted by the authors,
who take full responsibility for the content of this work.

\subsection*{Acknowledgements}
We thank Filip Najman for pointing out to us that everything reduces to the
Klein quartic.

\subsection*{Notation}
Throughout, $\zeta_7$ is a fixed primitive seventh root of unity,
$K = \Q(\zeta_7)^{+}$ is the totally real cyclic cubic field of discriminant
$49$, which has class number one, and $\sigma$ is the generator of $\Gal(K/\Q)$
induced by $\zeta_7 \mapsto \zeta_7^{2}$. We write $\Tr$ and $\Norm$ for the trace and norm
of $K/\Q$. Lower-case sans-serif letters $\x,\y,\w,\t,\dots$ are polynomial
variables, and the corresponding italic letters are their values. We write $f$
for one of the cubics in \cref{eq:the-cubics}, $F$ for its homogenization, and
$\theta \in K$ for the following root of $f$:
\begin{equation}
  \label{eq:thetas}
  \theta_{\spc} \colonequals 1 - (\zeta_7 + \zeta_7^{-1}),
  \qquad
  \theta_{\ns} \colonequals 3 + 2(\zeta_7 + \zeta_7^{-1}) = 5 - 2\theta_{\spc} .
\end{equation}
One has $\sigma(\theta_{\spc}) = -\theta_{\spc}^2 + 2\theta_{\spc} + 2$, the ring
of integers of $K$ is $\OO_K = \Z[\theta_{\spc}]$, and $\Z[\theta_{\ns}]$ has
index $8$ in $\OO_K$. The prime $2$ is inert in $K$ and $7$ is totally ramified.
Since $\theta_{\ns} - \sigma^{i}(\theta_{\ns}) = -2\bigl(\theta_{\spc} -
\sigma^{i}(\theta_{\spc})\bigr)$ and $f'(\theta) = \prod_{i=1,2}\bigl(\theta -
\sigma^{i}(\theta)\bigr)$,
\begin{equation}
  \label{eq:fprime}
  f_{\ns}'(\theta_{\ns}) = 4f_{\spc}'(\theta_{\spc}),
  \qquad
  \Norm\bigl(f_{\spc}'(\theta_{\spc})\bigr) = -\disc(f_{\spc}) = -7^{2} .
\end{equation}

\section{The equations}
\label{sec:equations}

This section rederives \cite{Furio--Lombardo26}*{Proposition 3.1}, the passage from
rational points on the modular curves
$X_{\ns}^{\sharp}(49)$ and $X_{\spc}^{\sharp}(49)$ to the superelliptic equations.

\subsection{Zywina's models}
By \cite{Zywina-15}*{Section 1.4, Theorem 1.5(ii)}, where $j_{\ns}$ and
$j_{\spc}$ below are the functions $J_6$ and $J_2$ (see also
\cite{Furio--Lombardo26}*{Theorem 2.1}), the $j$-maps of $X_{\ns}^{+}(7) \cong \Pone_{\t}$ and
$X_{\spc}^{+}(7) \cong \Pone_{\t}$ are given by the rational functions
\begin{align}
  \label{eq:jns}
  j_{\ns}(\t) &= \frac{64\,\t^{3}(\t^{2}+7)^{3}(\t^{2}-7\t+14)^{3}(5\t^{2}-14\t-7)^{3}}
               {(\t^{3}-7\t^{2}+7\t+7)^{7}}, \\
  \label{eq:jsp}
  j_{\spc}(\t) &= \frac{\t(\t+1)^{3}(\t^{2}-5\t+1)^{3}(\t^{2}-5\t+8)^{3}
                 (\t^{4}-5\t^{3}+8\t^{2}-7\t+7)^{3}}
               {(\t^{3}-4\t^{2}+3\t+1)^{7}} .
\end{align}
The first has degree $21$ and its three poles, the roots of
$f_{\ns}$, are the cusps of $X_{\ns}^{+}(7)$. The second has degree $28$; it has
poles of order $7$ at the roots of $f_{\spc}$ and at $t = \infty$, the four
cusps of $X_{\spc}^{+}(7)$, of which only $t = \infty$ is rational.
Homogenizing, $j_{\ns} = H_{\ns}^{3}/F_{\ns}^{7}$ and $j_{\spc} = \x
H_{\spc}^{3}/(\y F_{\spc})^{7}$, where $F_{\ns}$ and $F_{\spc}$ are the cubic
forms \eqref{eq:the-cubic-forms} and
\begin{align*}
  H_{\ns}(\x,\y) &\colonequals 4\x(\x^{2}+7\y^{2})(\x^{2}-7\x\y+14\y^{2})(5\x^{2}-14\x\y-7\y^{2}), \\
  H_{\spc}(\x,\y) &\colonequals (\x+\y)(\x^{2}-5\x\y+\y^{2})(\x^{2}-5\x\y+8\y^{2})
                   (\x^{4}-5\x^{3}\y+8\x^{2}\y^{2}-7\x\y^{3}+7\y^{4}) .
\end{align*}
The following elementary facts about these forms are used throughout.

\begin{lemma}
  \label{lem:arithmetic-of-forms}
  Let $x,y$ be coprime integers.
  \begin{enumerate}
  \item \label{item:gcd} The greatest common divisor of $F_{\ns}(x,y)$ and
    $H_{\ns}(x,y)$ divides $2^{21}7^{7}$, and that of $xH_{\spc}(x,y)$ and
    $yF_{\spc}(x,y)$ divides $7^{7}$. Moreover $\gcd(y,F(x,y)) = 1$ for both
    forms.
  \item \label{item:v7} $v_7(F(x,y)) \leq 1$ for both forms. Moreover:
    \begin{enumerate}
    \item \label{item:v7-ns} $7 \mid F_{\ns}(x,y)$ if and only if $7 \mid x$;
    \item \label{item:v7-sp} $7 \mid F_{\spc}(x,y)$ if and only if $7 \mid x+y$.
    \end{enumerate}
  \item \label{item:v2} $v_2(F_{\ns}(x,y)) \in \{0,3\}$, with value $3$ exactly
    when $x$ and $y$ are both odd; and $F_{\spc}(x,y)$ is odd.
  \item \label{item:substitution} $F_{\ns}(\x,\y) = F_{\spc}(\x-5\y,\,-2\y)$.
  \end{enumerate}
\end{lemma}
\begin{proof}
  One computes the resultants
  $\operatorname{res}_{\x}(F_{\ns},H_{\ns}) = \pm 2^{21}7^{7}\y^{21}$ and
  $\operatorname{res}_{\x}(F_{\spc},H_{\spc}) = \pm 7^{7}\y^{27}$. Both forms
  $F$ are congruent to $\x^{3}$ modulo $\y$, and $H_{\spc} \equiv \x^{9}$ modulo
  $\y$, while $F_{\spc} \equiv \y^{3}$ modulo $\x$; this gives the coprimality
  statements and \eqref{item:gcd}. Parts \eqref{item:v7} and \eqref{item:v2}
  are congruence computations. Call a pair $(x,y) \in (\Z/n\Z)^2$ \cdef{primitive}
  if it is not congruent to $(0,0)$ modulo any prime $p \mid n$; the reduction of
  a coprime pair of integers is primitive. Then $F(x,y) \equiv 0 \pmod{49}$ and
  $F_{\ns}(x,y) \equiv 0 \pmod{16}$ have no primitive solutions,
  $F_{\ns} \equiv \x^3$ and $F_{\spc} \equiv (\x+\y)^3 \pmod 7$,
  $F_{\ns} \equiv (\x+\y)^3 \pmod 2$, and $F_{\spc}(x,y) \equiv 1 \pmod 2$ for the three
  primitive pairs $(1,0)$, $(0,1)$, $(1,1)$ modulo $2$. That
  $v_2(F_{\ns}(x,y))$ is a multiple of $3$ also follows from the fact that $2$
  is inert in $K$. 
  Part \eqref{item:substitution} holds because
  $\theta_{\ns} = 5 - 2\theta_{\spc}$: indeed $x - \theta_{\ns}y = (x - 5y) - \theta_{\spc}(-2y)$.
\end{proof}

\subsection{The groups}
\label{sec:groups}
We recall \cite{Furio--Lombardo26}*{Definitions 2.6 and 2.7}. Let
$\varepsilon = 3$, the least quadratic non-residue modulo $7$. For
$\bullet \in \{\ns,\spc\}$, \cref{tab:groups} lists the Cartan subalgebra
$\mathfrak{c}_{\bullet}$ of $M_2(\F_7)$, the Cartan subgroup
$C_{\bullet}(7) = \mathfrak{c}_{\bullet} \cap \GL_2(\F_7)$ and its normalizer
$C_{\bullet}^{+}(7)$ in $\GL_2(\F_7)$, and a subspace $V_{\bullet}$ of
$M_2(\F_7)$. Let $N \colonequals I + 7M_2(\F_7)$ be the kernel of
reduction $\GL_2(\Z/49) \to \GL_2(\F_7)$. Whenever $\bullet$ is fixed, as in
the rest of this section, in \cref{prop:multiplicative-primes} and in
\cref{lem:criterion-groups}, we drop it and the level from the notation, and
write $\mathfrak{c}$, $C$, $C^{+}$, $V$ and $G^{\sharp}$ for
$\mathfrak{c}_{\bullet}$, $C_{\bullet}(7)$, $C_{\bullet}^{+}(7)$, $V_{\bullet}$
and $G_{\bullet}^{\sharp}(49)$.

\begin{table}[ht]
  \centering
  \tablesetup
  \begin{tabular}{IcIc:c:c:cI}
    \toprule
    \headrow
    $\bullet$ & $\mathfrak{c} = \mathfrak{c}_{\bullet}$ & $C = C_{\bullet}(7)$ & $C^+ = C_{\bullet}^{+}(7)$
      & $V = V_{\bullet}$ \\
    \headrule
    \padrows{5pt}
    $\ns$ & $\tabmatrix{a & \varepsilon b \\ b & a}$
      & $\tabmatrix{a & \varepsilon b \\ b & a}$
      & $C_{\ns}(7) \cup \tabmatrix{1 & 0 \\ 0 & -1} C_{\ns}(7)$
      & $\tabmatrix{a & \varepsilon b \\ -b & d}$ \\
    \rowsep
    $\spc$ & $\tabmatrix{a & 0 \\ 0 & d}$
      & $\tabmatrix{a & 0 \\ 0 & d}$
      & $C_{\spc}(7) \cup \tabmatrix{0 & 1 \\ 1 & 0} C_{\spc}(7)$
      & $\tabmatrix{a & b \\ c & a}$ \\
    \padrows{0pt}
    \bottomrule
  \end{tabular}
  \caption{The Cartan data at level $7$. Here, $C_{\bullet}(7)\subset \GL_2(\F_7)$ and
    $V_{\bullet} \subset M_2(\F_7)$.}
  \label{tab:groups}
\end{table}

The groups $G^{\sharp}$ of Furio and Lombardo are the subgroups of
$\GL_2(\Z/49)$ that reduce onto $C^{+}$ and meet $N$ in $I + 7V$. The
following lemma shows that such a subgroup exists and is unique up to
conjugation, and records its properties; the definition follows it.

\begin{lemma}
  \label{lem:groups}
  Let $\bullet \in \{\ns,\spc\}$.
  \begin{enumerate}
  \item \label{item:V} $V$ is the direct sum of the scalar matrices
    and of the orthogonal complement of $\mathfrak{c}$ for the trace
    pairing $(A,B) \mapsto \operatorname{tr}(AB)$. It has dimension $3$, it is
    stable under conjugation by $C^{+}$, and
    $V \cap \mathfrak{c}$ consists of the scalar matrices.
  \item \label{item:G-unique} There is a subgroup $G$ of $\GL_2(\Z/49)$ that
    reduces onto $C^{+}$ and meets $N$ in $I + 7V$, and
    any two such subgroups are conjugate by an element of $N$.
  \item \label{item:G-properties} Such a group $G$ has index $147$ if
    $\bullet = \ns$ and $196$ if $\bullet = \spc$, it has surjective
    determinant, and it contains $-I$.
  \end{enumerate}
\end{lemma}
\begin{proof}
  \eqref{item:V} The subalgebra $\mathfrak{c}_{\spc}$ is spanned by
  $\begin{psmallmatrix*}[r] 1 & 0 \\ 0 & 0\end{psmallmatrix*}$ and
  $\begin{psmallmatrix*}[r] 0 & 0 \\ 0 & 1\end{psmallmatrix*}$, and
  $\mathfrak{c}_{\ns}$ by $I$ and
  $J \colonequals \begin{psmallmatrix*}[r] 0 & \varepsilon \\ 1 & 0\end{psmallmatrix*}$.
  For $B = \begin{psmallmatrix*}[r] a & b \\ c & d\end{psmallmatrix*}$, the
  traces of these four products with $B$ are $a$, $d$, $a + d$ and
  $b + \varepsilon c$. So the orthogonal complement of $\mathfrak{c}_{\spc}$ is
  the space of matrices with zero diagonal, and that of $\mathfrak{c}_{\ns}$ the
  space of matrices
  $\begin{psmallmatrix*}[r] a & -\varepsilon c \\ c & -a\end{psmallmatrix*}$;
  neither contains a nonzero scalar, and adding the scalar matrices gives
  $V_{\spc}$ and (renaming $c$ as $-b$) $V_{\ns}$, which are thus the stated
  direct sums, of dimension $3$. Conjugation by $g \in C^{+}$ fixes the
  scalars, preserves the trace pairing and normalizes $C$, hence
  $\mathfrak{c}$ and its orthogonal complement; so it preserves
  $V$. Finally, a diagonal matrix in $V_{\spc}$ is scalar, and
  $\begin{psmallmatrix*}[r] a & \varepsilon b \\ b & a\end{psmallmatrix*} \in V_{\ns}$
  forces $b = -b = 0$.

  \eqref{item:G-unique} Modulo $49$ we have $(I + 7A)(I + 7B) = I + 7(A + B)$
  and $h(I + 7A)h^{-1} = I + 7\bar hA\bar h^{-1}$, with $\bar h$ the reduction
  of $h$. So $A \mapsto I + 7A$ is an isomorphism of $M_2(\F_7)$ onto $N$, and
  by \eqref{item:V} the subgroup $I + 7V$ is normal in the preimage
  $P$ of $C^{+}$ in $\GL_2(\Z/49)$. Put
  $Q \colonequals P/(I + 7V)$ and $\bar N \colonequals N/(I + 7V)$,
  a normal subgroup of $Q$ of order $7$ with $Q/\bar N \cong C^{+}$
  of order $72$ or $96$, prime to $7$. If $G$ is as in the statement, then
  $I + 7V \subseteq G \subseteq P$, and $H \colonequals G/(I + 7V)$
  is a complement of $\bar N$ in $Q$: $H \cap \bar N = (G \cap N)/(I + 7V)$
  is trivial, and $H$ maps onto $Q/\bar N$ because $G$ reduces onto
  $C^{+}$. Conversely, the preimage in $P$ of a complement of
  $\bar N$ is a subgroup as in the statement. By the Schur--Zassenhaus
  theorem, complements of $\bar N$ exist, and any two are conjugate in
  $Q = \bar N H$, hence by an element of $\bar N$; taking preimages, any two
  subgroups as in the statement are conjugate by an element of $N$.

  \eqref{item:G-properties} As $G \cap N = I + 7V$ has order $7^{3}$
  and $G/(G \cap N) \cong C^{+}$, we have
  $\#G = 7^{3}\,\#C^{+}$; since
  $\#\GL_2(\Z/49) = 7^{4}\,\#\GL_2(\F_7) = 7^{5} \cdot 288$, the index of $G$
  is $49 \cdot 288/72 = 196$ in the split case and $49 \cdot 288/96 = 147$ in
  the non-split case. Next, $\det(I + 7A) = 1 + 7\operatorname{tr}(A)$ modulo
  $49$ and $\operatorname{tr}(aI) = 2a$, so $\det(G)$ contains
  $\det(I + 7V) = 1 + 7\Z/49$, the kernel of
  $(\Z/49)^{\times} \to \F_7^{\times}$; and $\det(G)$ reduces onto
  $\det(C^{+}) = \F_7^{\times}$, since $\det$ is $ad$ on
  $C_{\spc}(7)$ and the norm $\F_{49}^{\times} \to \F_7^{\times}$ on
  $C_{\ns}(7)$. Hence $\det(G) = (\Z/49)^{\times}$. Finally, let $g \in G$
  reduce to $-I$ and write $g = -(I + 7B)$. Then
  $g^{2} = I + 14B \in G \cap N = I + 7V$, so $2B$, hence $B$, lies in
  $V$; thus $I + 7B \in G$ and $-I = g(I + 7B)^{-1} \in G$.
\end{proof}

\begin{definition}[\cite{Furio--Lombardo26}*{Definition 2.7}]
  \label{def:G-sharp}
  For $\bullet \in \{\ns,\spc\}$, let $G_{\bullet}^{\sharp}(49)$ be a subgroup of
  $\GL_2(\Z/49)$ that reduces onto $C^{+}$ and meets the kernel of reduction in
  $I + 7V$; by \cref{lem:groups} it is well defined up to
  conjugation, and so is the
  corresponding modular curve $X_{\bullet}^{\sharp}(49)$. These curves are
  $X_{\ns}^{\sharp}(49) =$ \mclabel{49.147.9.a.1} and $X_{\spc}^{\sharp}(49) =$
  \mclabel{49.196.9.a.1} \cite{Furio--Lombardo26}*{Remark 2.8}.
\end{definition}

There are exactly seven such subgroups, all conjugate by $N$. In the proof of
\cref{lem:groups}\eqref{item:G-unique} they correspond to the complements $H$
of $\bar N$ in $Q$. An element of $\bar N$ that normalizes $H$ centralizes it,
since $H \cap \bar N$ is trivial; but
$\bar N \cong M_2(\F_7)/V \cong \mathfrak{c}/\F_7 I$ (as $\mathfrak{c}$ meets its orthogonal complement
trivially), and the elements of $C^{+}$ outside $C$ act on it by $-1$. So the
seven $\bar N$-conjugates of $H$ are distinct.

\begin{lemma}
  \label{lem:nilpotents}
  The nonzero nilpotent elements of $V_{\spc}$ are the nonzero multiples of
  $\begin{psmallmatrix*}[r] 0 & 1 \\ 0 & 0\end{psmallmatrix*}$ and of
  $\begin{psmallmatrix*}[r] 0 & 0 \\ 1 & 0\end{psmallmatrix*}$. The space $V_{\ns}$
  contains no nonzero nilpotent element.
\end{lemma}
\begin{proof}
  A matrix $A \in M_2(\F_7)$ is nilpotent if and only if $\operatorname{tr}(A) =
  \det(A) = 0$. For $A = \begin{psmallmatrix*}[r] a & b \\ c & a\end{psmallmatrix*}
  \in V_{\spc}$ this says $2a = 0$ and $a^2 - bc = 0$, that is, $a = 0$ and $bc
  = 0$. For $A = \begin{psmallmatrix*}[r] a & \varepsilon b \\ -b &
  d\end{psmallmatrix*} \in V_{\ns}$ it says $d = -a$ and $-a^2 + \varepsilon b^2
  = 0$. Since $\varepsilon$ is not a square modulo $7$, the second equation
  forces $b = 0$; hence $a = 0$ and $d = 0$.
\end{proof}

\subsection{Primes of potentially multiplicative reduction}
The next proposition is the input from the arithmetic of elliptic curves. In the
non-split case it is a special case of \cite{Furio--Lombardo26}*{Proposition 2.10}, which
generalizes \cite{Furio-24}*{Proposition 5.13}; the proof below follows theirs.
In the split case we need to argue differently; see \cref{rem:FL-split}. The
following lemma isolates the modular input of the split case: where a point of
$X_{\spc}^{+}(7)$ with potentially multiplicative reduction reduces. Recall
that $C_{\spc}^{+}(7)$ is the stabilizer of the unordered pair of coordinate
lines of $\F_7^{2}$ and contains $-I$. So a point of $X_{\spc}^{+}(7)$ over a
field $k$ of characteristic $0$, with $j \neq 0, 1728$, is given by a pair
$(E,\{L_1,L_2\})$, where $E$ is an elliptic curve over $k$ and $\{L_1,L_2\}$ is
a $\Gal_k$-stable unordered pair of distinct lines in $E[7](\bar k)$ (a
\cdef{level structure} on $E$), up to isomorphism and quadratic twist \cite{RSZB-22}*{Section 2.3}. This is the
moduli interpretation of the curve $X_{\mathrm{split}}(7)$ of Mazur
\cite{Mazur-77}, which is our $X_{\spc}^{+}(7)$; similarly, Mazur's
$X_{\mathrm{nonsplit}}(7)$ is our $X_{\ns}^{+}(7)$.

\begin{lemma}
  \label{lem:cusp-reduction}
  Let $p \neq 7$ be a prime, and let $E$ be an elliptic curve over $\Q_p$ with
  potentially multiplicative reduction, so that $E$ is a quadratic twist of a
  Tate curve $E_q$ and $E[7](\Qbar_p) = E_q[7](\Qbar_p)$ contains the line
  $\mu_7(\Qbar_p)$. Let $\{L_1,L_2\}$ be a $\Gal_{\Q_p}$-stable pair of distinct
  lines in $E[7](\Qbar_p)$, and let $t \in X_{\spc}^{+}(7)(\Q_p)$ be the point
  defined by $(E,\{L_1,L_2\})$. If $\mu_7(\Qbar_p) \in \{L_1,L_2\}$, then $t$
  reduces modulo $p$ to the rational cusp $t = \infty$: writing $t = (x:y)$
  with $x, y \in \Z_p$ not both in $p\Z_p$, we have $p \mid y$.
\end{lemma}
\begin{proof}
  We first describe the cusps. They correspond to the orbits of the group of
  matrices $\pm\begin{psmallmatrix*}[r] 1 & \ast \\ 0 & 1\end{psmallmatrix*}$
  on the $28$ unordered pairs $\{L_1,L_2\}$ of distinct lines in $\F_7^2$, in
  a basis whose first vector spans the subgroup scheme $\mu_7$ of the Tate
  curve $E_q$ over $\Z[1/7]((q))$
  \cite{DR-73}*{VI, 5.1 and Construction 5.3, and VII, 1.13--1.16},
  \cite{KM-85}*{Corollary 10.2.5}. Indeed, the $7$-torsion of the Tate curve
  has the basis $(\zeta_7, q^{1/7})$ over $\Z[1/7][\zeta_7]((q^{1/7}))$. The
  automorphism of this ring over $\Z[1/7][\zeta_7]((q))$ that fixes $\zeta_7$
  and sends $q^{1/7}$ to $\zeta_7 q^{1/7}$ acts on this basis by
  $\begin{psmallmatrix*}[r] 1 & 1 \\ 0 & 1\end{psmallmatrix*}$, which generates
  the unipotent matrices, and $-I$ comes from the automorphism $-1$ of $E_q$.
  The seven pairs containing the line $\mu_7$
  form one orbit, and the other $21$ pairs form three. The group $\Gal_{\Q}$ acts
  on the cusps through the cyclotomic character, by the matrices
  $\begin{psmallmatrix*}[r] a & 0 \\ 0 & 1\end{psmallmatrix*}$ with
  $a \in \F_7^{\times}$ \cite{DR-73}*{VI, 5.2--5.3}. These fix the first orbit
  and permute the other three transitively, so the first orbit is the unique
  rational cusp. In Zywina's coordinate $\t$ this cusp is $t = \infty$, and the
  other three are the roots of $f_{\spc}$, the remaining poles of $j_{\spc}$ in
  \cref{eq:jsp}.

  Next, the model. Over $\Z[1/7]$ the coarse moduli scheme of
  $X_{\spc}^{+}(7)$ is smooth \cite{DR-73}*{VI, Proposition 6.7},
  \cite{KM-85}*{Theorem 10.9.7}, and proper, being finite over the $j$-line;
  its cusps form a finite \'etale subscheme \cite{KM-85}*{Theorem 10.9.4(3)}.
  Let $\mathcal{X}$ be its base change to $\Z_p$. By properness the rational
  cusp extends to a section of $\mathcal{X}$, and a smooth proper curve over
  $\Z_p$ with fibers of genus $0$ and a section is isomorphic to
  $\Pone_{\Z_p}$. So $\mathcal{X} = \Pone_{\Z_p}$ in some coordinate $\t'$, and
  the cusps have pairwise distinct reductions in $\t'$. They also have pairwise
  distinct reductions in Zywina's coordinate $\t$: the roots of the monic cubic
  $f_{\spc}$ are $p$-integral, so they do not reduce to $\infty$, and their
  reductions are distinct because $\disc(f_{\spc}) = 7^2$. We show that two
  such coordinates differ by an element of $\PGL_2(\Z_p)$. The rational cusp is
  a $\Z_p$-point in $\t'$, and $\PGL_2(\Z_p)$ acts transitively on
  $\Pone(\Z_p)$, so after composing $\t'$ with an element of $\PGL_2(\Z_p)$ we
  may assume that the rational cusp is $\t' = \infty$. Then $\t' = a\t + b$
  with $a \in \Q_p^{\times}$ and $b \in \Q_p$, and the other three cusps are
  $a\theta_i + b$, for the roots $\theta_i$ of $f_{\spc}$. These are
  $p$-integral, since they do not reduce to $\infty$, and their differences
  $a(\theta_i - \theta_j)$ are units, since their reductions are distinct. Since
  $\theta_i - \theta_j$ is a unit, so is $a$, and $b = (a\theta_1 + b) -
  a\theta_1$ is integral. Hence $\mathcal{X} = \Pone_{\Z_p}$ in the coordinate
  $\t$ as well, and reduction modulo $p$ of $X_{\spc}^{+}(7)(\Q_p)$ is
  computed in $\t$.

  Finally, the reduction of $t$. Since $-I \in C_{\spc}^{+}(7)$, the point $t$
  depends only on the pair of subgroups $\{L_1,L_2\}$ of $E[7]$ and on $E$ up
  to quadratic twist; so we may assume that $E = E_q$, and twisting does not
  move the line $\mu_7$. Since $v_p(j(E)) < 0$, the point $t$ reduces to a cusp
  $\mathfrak{c}$ of the special fiber of $\mathcal{X}$. The formal completion of
  $\mathcal{X}$ along $\mathfrak{c}$ is described by the Tate curve: the points
  of $X_{\spc}^{+}(7)(\Q_p)$ that reduce to $\mathfrak{c}$ are the points
  $(E_{q'},\{L_1',L_2'\})$ with $v_p(q') > 0$ and $\{L_1',L_2'\}$ in the orbit
  $\mathfrak{c}$, the lines being read in the basis $(\zeta_7, q'^{1/7})$ of
  $E_{q'}[7]$ \cite{DR-73}*{VII, Corollaire 2.2}, \cite{KM-85}*{Theorem
    8.11.10}. This representation of a point is unique up to the automorphism
  $-1$: the point determines $q'$, since $j(E_{q'})$ is its $j$-invariant, and it
  determines the pair of lines up to $\Aut(E_{q'}) = \{\pm 1\}$, since $j \neq 0,
  1728$; and $-1$ fixes every line. Applied to $t = (E_q, \{L_1,L_2\})$, this
  says that $\{L_1,L_2\}$ lies in the orbit $\mathfrak{c}$; so $\mathfrak{c}$ is
  the rational cusp if and only if $\mu_7(\Qbar_p) \in \{L_1,L_2\}$.
\end{proof}

\begin{proposition}
  \label{prop:multiplicative-primes}
  Let $\bullet \in \{\ns,\spc\}$, and let $E$ be an elliptic curve over $\Q$ with a
  basis $(P_1, P_2)$ of $E[49](\Qbar)$ with respect to which
  $\rho_{E,49}(\Gal_{\Q}) \subseteq G_{\bullet}^{\sharp}(49)$. The basis
  $(7P_1, 7P_2)$ of $E[7](\Qbar)$, with respect to which
  $\rho_{E,7}(\Gal_{\Q}) \subseteq C_{\bullet}^{+}(7)$, defines a rational point
  $t$ of $X_{\bullet}^{+}(7)$ with $j_{\bullet}(t) = j(E)$. In the split case,
  $t$ is the point $(E, \{L_1,L_2\})$, where $L_1, L_2 \subset E[7](\Qbar)$ are
  the lines spanned by $7P_1$ and $7P_2$; the pair $\{L_1,L_2\}$ is
  $\Gal_{\Q}$-stable. Write $t = (x:y)$ with $x,y$ coprime integers. Then
  $v_p(F_{\bullet}(x,y)) \equiv 0 \pmod 7$ for every prime $p \notin \{2,7\}$. (In the split
  case this also holds for $p = 2$, since $F_{\spc}(x,y)$ is odd by
  \cref{lem:arithmetic-of-forms}\eqref{item:v2}.)
\end{proposition}
\begin{proof}
  Fix such a prime $p$ and suppose that
  \[
    m \colonequals v_p(F_{\bullet}(x,y)) > 0 .
  \]
  We show that $7 \mid m$.

  \emph{Step 1: $E$ has potentially multiplicative reduction at $p$, with
    Tate parameter $q$ satisfying $v_p(q) = 7m$.} By
  \cref{lem:arithmetic-of-forms}\eqref{item:gcd}, $p \nmid H_{\ns}(x,y)$ in
  the non-split case, and $p \nmid xyH_{\spc}(x,y)$ in the split case. The
  formulas $j_{\ns} = H_{\ns}^{3}/F_{\ns}^{7}$ and
  $j_{\spc} = \x H_{\spc}^{3}/(\y F_{\spc})^{7}$ then give
  \[
    v_p(j(E)) = -7m < 0,
  \]
  so $E$ has potentially multiplicative reduction at $p$. By Tate's theory
  \cite{Silverman-94}*{Chapter V, Lemma 5.1 and Theorem 5.3}, there is a
  unique $q \in \Q_p^{\times}$ with $v_p(q) > 0$ and $j(E_q) = j(E)$, and a
  character $\chi \colon \Gal_{\Q_p} \to \{\pm 1\}$ such that $E_{\Q_p}$ is the
  quadratic twist of the Tate curve $E_q$ by $\chi$. Since
  $j(E_q) = q^{-1} + 744 + 196884\,q + \cdots$,
  \[
    v_p(q) = -v_p(j(E)) = 7m .
  \]

  \emph{Step 2: the image of inertia contains the matrix
    $\begin{psmallmatrix*} 1 & 14m \\ 0 & 1\end{psmallmatrix*}$, in a basis of
    $E[49](\Qbar_p)$ whose first vector spans $\mu_{49}(\Qbar_p)$.} Tate's
  uniformization $E_q(\Qbar_p) \cong \Qbar_p^{\times}/q^{\Z}$ is
  $\Gal_{\Q_p}$-equivariant, so the classes of $\zeta_{49}$ and $q^{1/49}$ form a
  basis of $E_q[49](\Qbar_p)$. Write $q = p^{7m}u$ with $u \in \Z_p^{\times}$, and let
  $c \colon I_p \to \Z/49$ be the surjective character of the inertia group
  defined by $\tau(p^{1/49}) = \zeta_{49}^{c(\tau)}p^{1/49}$. Since $p \neq 7$, both
  $\zeta_{49}$ and $u^{1/49}$ lie in an unramified extension of $\Q_p$, so for
  $\tau \in I_p$, we have that $\tau(\zeta_{49}) = \zeta_{49}$ and
  $\tau(q^{1/49}) = \zeta_{49}^{7m\,c(\tau)}\,q^{1/49}$. That is,
  \[
    \rho_{E,49}(\tau) = \chi(\tau)\begin{pmatrix*} 1 & 7m c(\tau) \\ 0 & 1 \end{pmatrix*}
  \]
  in this basis. Choose $\tau$ with $c(\tau) = 1$. Since $\chi(\tau)^2 = 1$,
  \[
    \rho_{E,49}(\tau^2) = \begin{pmatrix*} 1 & 14m \\ 0 & 1 \end{pmatrix*}
    = I + 14m N_0,
    \qquad
    N_0 \colonequals \begin{pmatrix*}[r] 0 & 1 \\ 0 & 0 \end{pmatrix*} .
  \]

  \emph{Step 3: $m\bar N \in V$ for a nonzero nilpotent
    $\bar N \in M_2(\F_7)$ whose image is the line $\mu_7(\Qbar_p)$.} Identify
  $E[49](\Qbar)$ with $E[49](\Qbar_p)$ through a fixed embedding of $\Qbar$
  into $\Qbar_p$, and let $B \in \GL_2(\Z/49)$ be the matrix of the
  change from the basis of Step~2 to the basis of the statement, in which
  $\rho_{E,49}(\Gal_{\Q}) \subseteq G^{\sharp}$. In the latter
  basis
  \[
    \rho_{E,49}(\tau^2) = B(I + 14mN_0)B^{-1} = I + 7\cdot 2m\bar N,
    \qquad
    \bar N \colonequals BN_0B^{-1} \bmod 7 .
  \]
  The matrix $\bar N$ is nonzero and nilpotent, and its image is spanned by the
  class of $\zeta_7$, i.e.,~it is the line $\mu_7(\Qbar_p) \subset
  E[7](\Qbar_p)$. Since $\rho_{E,49}(\tau^2)$ lies in the kernel of reduction of
  $G^{\sharp}$, which is $I + 7V$, we get
  $2m\bar N \in V$; hence $m\bar N \in V$.

  \emph{Step 4: if $7 \nmid m$, then $\bullet = \spc$ and $\mu_7(\Qbar_p)$ is one
    of the two lines of the level structure of $t = x/y$.} If $7 \nmid m$, then
  $\bar N \in V$ is a nonzero nilpotent matrix. By
  \cref{lem:nilpotents} this is impossible when $\bullet = \ns$. When $\bullet =
  \spc$, the same lemma shows that $\bar N$ is a multiple of
  $\begin{psmallmatrix*}[r] 0 & 1 \\ 0 & 0\end{psmallmatrix*}$ or of
  $\begin{psmallmatrix*}[r] 0 & 0 \\ 1 & 0\end{psmallmatrix*}$, so its image
  $\mu_7(\Qbar_p)$ is one of the two coordinate lines $L_1, L_2 \subset
  E[7](\Qbar)$, the two lines of the level structure defining $t = x/y$.

  \emph{Step 5: in the situation of Step~4, $p \mid y$; this is a contradiction.}
  By Step~4 and \cref{lem:cusp-reduction}, applied to $E_{\Q_p}$ and the pair
  $\{L_1,L_2\}$, the point $t = x/y$ reduces modulo $p$ to the rational cusp,
  that is, $p \mid y$. But $p \mid F_{\spc}(x,y)$ and $\gcd(y, F_{\spc}(x,y)) = 1$
  by \cref{lem:arithmetic-of-forms}\eqref{item:gcd}. Hence $7 \mid m$.
\end{proof}

\begin{remark}
  \label{rem:FL-split}
  We explain why the split case is treated differently. The proof of
  \cite{Furio--Lombardo26}*{Proposition 2.10} is local: for an elliptic curve over a number
  field $L$ and a prime $\mathfrak{p}$ of $L$ of potentially multiplicative
  reduction, it first adjoins $\zeta_7$ to $L$ and then considers the image
  of the whole local Galois group of the completion $L(\zeta_7)_{\mathfrak{P}}$
  at a prime $\mathfrak{P}$ above $\mathfrak{p}$. In the split
  case it uses an element $\tau$ of that group with $\chi_{49}(\tau) = 1 + 7$,
  where $\chi_{49}$ is the mod-$49$ cyclotomic character, whose existence is
  attributed there to the tameness of $7$ in $L$. Such an element exists if
  and only if $\zeta_{49} \notin L(\zeta_7)_{\mathfrak{P}}$, a condition on
  $\mathfrak{p}$ rather than on the primes above $7$. For $L = \Q$ it fails
  exactly at the primes $p$ with $p^{6} \equiv 1 \pmod{49}$, among them the
  primes $p \equiv \pm 1 \pmod{49}$, and such primes can divide
  $F_{\spc}(x,y)$.
  \Cref{prop:multiplicative-primes} recovers the congruence
  $v_p(F_{\spc}(x,y)) \equiv 0 \pmod 7$ at every prime $p \neq 7$ using only
  inertia, together with the position of the rational cusp
  (\cref{lem:cusp-reduction}). The same analysis shows that the rational cusp
  $t = \infty$ of $X_{\spc}^{+}(7)$ is \emph{unramified} in
  $X_{\spc}^{\sharp}(49) \to X_{\spc}^{+}(7)$, while the other three cusps are
  totally ramified: the cusp widths of $X_{\spc}^{\sharp}(49)$ are $7$ (seven
  times) and $49$ (three times), and those of $X_{\ns}^{\sharp}(49)$ are $49$
  (three times). Accordingly, the further assertion of the split case of
  \cite{Furio--Lombardo26}*{Proposition 3.1}, that $y$ is a seventh power, is established
  by their argument only at the primes $p \mid y$ with $p^{6} \not\equiv 1
  \pmod{49}$. Neither \cite{Furio--Lombardo26} nor the present paper uses that assertion.
\end{remark}

\begin{proposition}[\cite{Furio--Lombardo26}*{Proposition 3.1}]
  \label{prop:FL}
  Let $E$ be an elliptic curve over $\Q$ without complex multiplication.
  \begin{enumerate}
  \item If $\rho_{E,49}(\Gal_{\Q})$ is conjugate to a subgroup of
    $G_{\ns}^{\sharp}(49)$, there are coprime integers $x,y$ and an integer $w$
    with $j(E) = j_{\ns}(x/y)$ and $F_{\ns}(x,y) = kw^7$, where $k \in \{7,56\}$.
  \item If $\rho_{E,49}(\Gal_{\Q})$ is conjugate to a subgroup of
    $G_{\spc}^{\sharp}(49)$, there are coprime integers $x,y$ and an integer $w$
    with $j(E) = j_{\spc}(x/y)$ and $F_{\spc}(x,y) = kw^7$, where $k \in \{1,7\}$.
  \end{enumerate}
\end{proposition}
\begin{proof}
  Let $t = x/y$ be the rational point of $X_{\bullet}^{+}(7)$ defined by $E$ and
  its level structure, as in \cref{prop:multiplicative-primes}; then $j(E) =
  j_{\bullet}(t)$, because $j_{\bullet}$ is the $j$-map of $X_{\bullet}^{+}(7)$ in
  Zywina's coordinate \cite{Zywina-15}*{Section 1.4}, \cite{Furio--Lombardo26}*{Theorem 2.1}. By
  \cref{prop:multiplicative-primes} and
  \cref{lem:arithmetic-of-forms}\eqref{item:v7}--\eqref{item:v2}, and since $-1$
  is a seventh power, $F_{\bullet}(x,y) = kw^7$ with $k \in \{1,7,8,56\}$ when
  $\bullet = \ns$ and $k \in \{1,7\}$ when $\bullet = \spc$. In the non-split case,
  \cite{Furio--Lombardo26}*{Lemma 2.11} shows that $v_7(j(E)) > 0$. If $7 \nmid x$ then $7
  \nmid F_{\ns}(x,y)$ and $H_{\ns}(x,y) \equiv 20x^{7} \not\equiv 0 \pmod 7$, so
  $v_7(j(E)) = 0$. Hence $7 \mid x$, and $7 \mid k$.
\end{proof}

\section{Kummer descent}
\label{sec:descent}

We follow Bruin \cite{Bruin02-thesis}*{Section 3.1}, \cite{Bruin06-nn2}*{Section 3}. Throughout this
section we fix $\bullet \in \{\ns,\spc\}$, write $f = f_{\bullet}$,
$F = F_{\bullet}$, $\theta = \theta_{\bullet}$, and fix one of the values of
$k$ allowed by \cref{tab:descent-data}. The table also records $\disc(f)$ and
the set $\calS$ of \cdef{bad primes}, those dividing $k\disc(f)$. Its last
two columns give generators of the group $\OKS^{\times}$ of $\calS$-units of
$K$, and the order of the Selmer group $K(7,\calS)$, which is defined below.

\begin{table}[ht]
  \centering
  \tablesetup
  \begin{tabular}{IcIc:c:c:c:c:cI}
    \toprule
    \headrow
    $\bullet$ & $f$ & $\disc(f)$ & $k$ & $\calS$ & $\OKS^{\times}$ & $\#K(7,\calS)$ \\
    \headrule
    $\ns$ & $\t^3 - 7\t^2 + 7\t + 7$ & $2^{6}7^{2}$ & $1,\ 7,\ 8,\ 56$ & $\{2,7\}$
      & $\pm\langle \varepsilon_1, \varepsilon_2, \varpi, 2\rangle$ & $7^{4}$ \\
    \rowsep
    $\spc$ & $\t^3 - 4\t^2 + 3\t + 1$ & $7^{2}$ & $1,\ 7$ & $\{7\}$
      & $\pm\langle \varepsilon_1, \varepsilon_2, \varpi\rangle$ & $7^{3}$\\
    \bottomrule
  \end{tabular}
  \caption{The data of the descent. In both rows $\varepsilon_1 \colonequals
    \theta_{\spc} - 1$ and $\varepsilon_2 \colonequals 2 - \theta_{\spc}$ are
    units, and $\varpi \colonequals 1 + \theta_{\spc}$ has norm $7$.}
  \label{tab:descent-data}
\end{table}

\subsection{The covering curves}
Since $K$ has class number one, we may define the \cdef{$(7,\calS)$-Selmer group} of
\cite{Bruin06-nn2}*{Section 3} as
\[
  K(7,\calS) \colonequals \OKS^{\times}/\OKS^{\times 7},
\]
of the order given in \cref{tab:descent-data}: by Dirichlet's unit theorem,
$\OKS^{\times}$ has rank $2 + \#\calS$, since $2$ is inert and $7$ is totally
ramified in $K$, and its torsion is $\pm 1$; generators are listed in
\cref{tab:descent-data}. We fix a set of representatives
in $\OKS^{\times}$ for $K(7,\calS)$, and we let $\Delta(\theta,k)$ be the subset of those
representatives $\delta$ for which
\[
  \Norm(\delta)/k \in \Q^{\times 7} ;
\]
this condition only depends on the class of $\delta$, since $\Norm(\nu^7) \in
\Q^{\times 7}$. The subset $\Delta(\theta,k)$ has $49$ elements for each of the six pairs
$(F,k)$. Let $\sfa_0,\sfa_1,\sfa_2$ be coordinates on $\PP^2_{\Q}$, and put
\begin{align*}
  \sfu_0 &\colonequals \sfa_0 + \sfa_1\theta + \sfa_2\theta^{2}, \\
  \sfu_1 &\colonequals \sfa_0 + \sfa_1\sigma(\theta) + \sfa_2\sigma(\theta)^{2}, \\
  \sfu_2 &\colonequals \sfa_0 + \sfa_1\sigma^{2}(\theta) + \sfa_2\sigma^{2}(\theta)^{2},
\end{align*}
in $K[\sfa_0,\sfa_1,\sfa_2]$, with the index $i$ of $\sfu_i$ read modulo $3$.
Letting $\sigma$ act on $K[\sfa_0,\sfa_1,\sfa_2]$ through the coefficients, we
have $\sigma(\sfu_i) = \sfu_{i+1}$; and for $\omega = a_0 + a_1\theta +
a_2\theta^2 \in K$ we have $\sfu_i(a_0,a_1,a_2) = \sigma^{i}(\omega)$. Thus
$\sfu_0$ is the \cdef{generic element} of $K$, and $\sfu_0\sfu_1\sfu_2$ is its norm; we
extend $\Tr$ and $\Norm$ to $K[\sfa_0,\sfa_1,\sfa_2]$ coefficientwise. For
$\delta \in K^{\times}$ the \cdef{coordinate forms}
$P^{(n)}_{\theta,\delta} \in \Q[\sfa_0,\sfa_1,\sfa_2]$ of degree $7$ are
defined by
\begin{equation}
  \label{eq:P-forms}
  \delta\,(\sfa_0+\sfa_1\theta + \sfa_2\theta^2)^{7}
  = P^{(0)}_{\theta,\delta} + P^{(1)}_{\theta,\delta}\,\theta + P^{(2)}_{\theta,\delta}\,\theta^{2},
\end{equation}
and we put
\begin{equation}
  \label{eq:C-and-phi}
  C_{\theta,\delta} \colonequals \bigl\{P^{(2)}_{\theta,\delta} = 0\bigr\} \subset \PP^2_{\Q} .
\end{equation}
Writing $\mathbf{a} = (a_0:a_1:a_2)$ for a point of $C_{\theta,\delta}$, we
define
\begin{equation}
  \label{eq:phi}
  \phi_{\theta,\delta} \colon C_{\theta,\delta} \to \Pone_{\Q},
  \qquad \mathbf{a} \mapsto
  \bigl(P^{(0)}_{\theta,\delta}(\mathbf{a}) : -P^{(1)}_{\theta,\delta}(\mathbf{a})\bigr).
\end{equation}
We identify $\PP^2(\Q)$ with $K^{\times}/\Q^{\times}$ through the bijection
\begin{equation}
  \label{eq:P2=K*/Q*}
  \PP^2(\Q) \to K^{\times}/\Q^{\times}, \qquad \mathbf{a} = (a_0:a_1:a_2) \mapsto [a_0 + a_1\theta + a_2\theta^2] .
\end{equation}

\begin{remark}[Intrinsic description]
  \label{rem:intrinsic}
  Write $\calV_K$ for $K$ regarded as a $\Q$-vector space. For a
  finite-dimensional $\Q$-vector space $W$, let
  $\PP(W) \colonequals \Proj \Sym(W^{\vee})$ be the projective space of lines
  in $W$, so that $\PP(W)(\Q) = (W \setminus \{0\})/\Q^{\times}$; in
  particular $\PP(\calV_K)(\Q) = K^{\times}/\Q^{\times}$. The coordinates
  $\sfa_0, \sfa_1, \sfa_2$ form the basis of $\calV_K^{\vee}$ dual to
  $1, \theta, \theta^2$, so $\Sym(\calV_K^{\vee}) = \Q[\sfa_0,\sfa_1,\sfa_2]$,
  and the bijection \cref{eq:P2=K*/Q*} is the resulting isomorphism
  $\PP^2_{\Q} \cong \PP(\calV_K)$. In these terms the objects of
  the descent do not depend on coordinates:
  \[
    C_{\theta,\delta} = \bigl\{[\omega] \in \PP(\calV_K) : \delta\omega^{7} \in \Q + \Q\theta\bigr\},
    \qquad
    \phi_{\theta,\delta}([\omega]) = [\delta\omega^{7}] \in \PP(\Q + \Q\theta),
  \]
  and $\PP(\Q + \Q\theta)$ is identified with $\Pone_{\t}$ by $x - \theta y
  \mapsto (x : y)$ (\cref{lem:sign} below). Since $\Q + \Q\theta_{\ns} = \Q +
  \Q\theta_{\spc}$, the curves are the same for both roots, and only the
  coordinate on the target changes; see \cref{lem:one-family}. Likewise, in
  \cref{sec:klein} the quotient map and its target will be
  $\psi([\omega]) = [\omega^{3}\sigma^{-1}(\omega)]$ and $Z_{\theta,\delta} =
  \{[v] \in \PP(\calV_K) : \Tr(\eta\, v^{3}\sigma(v)) = 0\}$, with $\eta =
  \delta/f'(\theta)$. The coordinate forms $P^{(n)}_{\theta,\delta}$ of
  \cref{eq:P-forms} and $Q^{(n)}_{\theta,\delta}$ of \cref{eq:Q-forms} are the
  equations of these objects in the basis
  $1, \theta, \theta^2$; we keep them because they are what the computations
  use.
\end{remark}

Finally, let $\mcU$ be the complement of $\{\x = \y = 0\}$ in the hypersurface
$\{F(\x,\y) = k\,\w^7\} \subset \A^3_{\Z}$, so that
\[
  \mcU(\Z) = \bigl\{(x,y,w) \in \Z^3 : F(x,y) = kw^7 \text{ and } \gcd(x,y) = 1\bigr\},
\]
and let
\[
  \pi \colon \mcU_{\Q} \to \Pone_{\Q}, \qquad (x,y,w) \mapsto (x:y) .
\]

\begin{lemma}[Choice of sign]
  \label{lem:sign}
  Let $\delta, \omega \in K^{\times}$ and $x, y \in \Q$ with
  $x - \theta y = \delta\omega^7$. Then $[\omega] \in C_{\theta,\delta}(\Q)$ and
  $\phi_{\theta,\delta}([\omega]) = (x:y)$. Thus the common target of $\pi$ and
  $\phi_{\theta,\delta}$ is the $\Pone$ with Zywina's coordinate $\t = \x/\y$.
\end{lemma}
\begin{proof}
  By \cref{eq:P-forms}, the coefficients of $1, \theta, \theta^2$ in
  $\delta\omega^7$ are the values at the coordinates of $\omega$ of
  $P^{(0)}_{\theta,\delta}$, $P^{(1)}_{\theta,\delta}$ and
  $P^{(2)}_{\theta,\delta}$. So $P^{(2)}_{\theta,\delta} = 0$,
  $P^{(0)}_{\theta,\delta} = x$ and $P^{(1)}_{\theta,\delta} = -y$, and
  $\phi_{\theta,\delta}([\omega]) = (x : y)$; when $(x,y,w) \in \mcU(\Z)$ this is
  $\pi(x,y,w)$.
\end{proof}

\begin{proposition}[{\cite{Bruin06-nn2}*{Theorem 3}, \cite{Bruin02-thesis}*{Lemma 3.1.2}}]
  \label{prop:bruin}
  With notation as above,
  \[
    \pi\bigl(\mcU(\Z)\bigr) \subseteq \bigcup_{\delta \in \Delta(\theta,k)}
    \phi_{\theta,\delta}\bigl(C_{\theta,\delta}(\Q)\bigr).
  \]
  More precisely, for $(x,y,w) \in \mcU(\Z)$ there are $\delta \in \Delta(\theta,k)$ and
  $\omega \in K^{\times}$ with $x - \theta y = \delta\omega^7$; then $[\omega] \in
  C_{\theta,\delta}(\Q)$ and $\phi_{\theta,\delta}([\omega]) = (x:y)$.
\end{proposition}
\begin{proof}
  Let $(x,y,w) \in \mcU(\Z)$. Let $\mfp$ be a prime of $\OO_K$ not above
  $\calS$, and let $p$ be the rational prime below it; $p$ is unramified in
  $K$. If $\mfp$ divides both $x - \theta y$ and a conjugate
  $x - \sigma^{i}(\theta)y$ with $i \neq 0$, then it divides
  $(\theta - \sigma^{i}(\theta))x$ and $(\theta - \sigma^{i}(\theta))y$, hence, since
  $\gcd(x,y) = 1$, the element $\theta - \sigma^{i}(\theta)$, which divides
  $\disc(f)$; this is a contradiction. Applying powers of $\sigma$, we see that
  the three ideals $\sigma^{i}(x - \theta y)\OKS$ are pairwise coprime. Their
  product is $F(x,y)\OKS = kw^7\OKS = w^7\OKS$, since $k$ is an $\calS$-unit, so
  each of them is a seventh power. Since $K$ has class number one,
  $x - \theta y = \delta\omega^7$ with $\delta \in \OKS^{\times}$, which we may take among the fixed
  representatives. Taking norms, $\Norm(\delta)/k = (w/\Norm(\omega))^{7}$, so
  $\delta \in \Delta(\theta,k)$. Finally, the coefficient of $\theta^2$ in
  $\delta\omega^7 = x - \theta y$ is $0$, so
  $[\omega] \in C_{\theta,\delta}(\Q)$, and $\phi_{\theta,\delta}([\omega]) = (x:y)$ by \cref{lem:sign}.
\end{proof}

\subsection{The trace form}
Since $K/\Q$ is Galois, a point $\omega \in K$ can be read through its three
conjugates $\sigma^{i}(\omega) = \sfu_i(a_0,a_1,a_2)$. Accordingly, the linear
forms $\sfu_i$ are coordinates on $\PP^2$ over $K$, and Galois acts on them,
through the coefficients, by cyclic permutation. Indeed, if $\tau \in \Gal_{\Q}$ restricts to
$\sigma^{a}$ on $K$, then $\tau(\sfu_i) = \sfu_{i+a}$, indices modulo $3$. In these
coordinates the trace becomes a diagonal form, and the curves
$C_{\theta,\delta}$ become twists of the Fermat septic
\[
  \calF \colon \sfu_0^7 + \sfu_1^7 + \sfu_2^7 = 0 .
\]
Throughout, we put $\eta \colonequals \delta/f'(\theta)$, and we write
$\phi_0 \colon C_0 \to \Pone_{\Q}$ for $\phi_{\theta,f'(\theta)}$, the subscript
$0$ referring to the \cdef{trivial twist} $\eta = 1$; by
\cref{lem:trace-form} below, $C_0 = C_{\theta,f'(\theta)}$ is the curve
$\{\Tr(\omega^{7}) = 0\}$. Here and below, a curve written as
$\{\Tr(\eta\omega^{7}) = 0\}$ is described by the condition on its points $[\omega]$,
and its equation is $\Tr(\eta\,\sfu_0^{7}) = 0$. The
multiplicative map $\omega \mapsto \omega^{3}\sigma^{-1}(\omega)$ commutes with
$\sigma$, hence is defined over $\Q$; it sends $\{\Tr(\eta\omega^7) = 0\}$ to a
twist $\{\Tr(\eta\,v^{3}\sigma(v)) = 0\}$ of the Klein quartic, and this is the
subject of \cref{sec:klein}.

\begin{lemma}
  \label{lem:trace-form}
  For $\delta \in K^{\times}$,
  \[
    P^{(2)}_{\theta,\delta} = \Tr\bigl(\delta\,\sfu_0^{7}/f'(\theta)\bigr).
  \]
\end{lemma}
\begin{proof}
  By Euler's formula \cite{Neukirch-99}*{Chapter III, Proposition 2.4},
  $\Tr(\theta^{i}/f'(\theta))$ equals $0$ for $i = 0,1$ and $1$ for $i = 2$. So
  $\Tr(\,\cdot\,/f'(\theta))$ is the coefficient of $\theta^2$ in the power
  basis.
\end{proof}

The covering $\phi_{\theta,\delta}$ is geometrically Galois \cite{Bruin02-thesis}*{Theorem
  3.1.1}: the group $\Aut(\phi_{\theta,\delta,\Qbar})$ of geometric automorphisms is
$\mu_7(\Qbar)^3/\mu_7(\Qbar)$, acting by
\begin{equation}
  \label{eq:aut-phi}
  (\sfu_0:\sfu_1:\sfu_2) \mapsto (\zeta_0\sfu_0:\zeta_1\sfu_1:\zeta_2\sfu_2).
\end{equation}
Over $\Q$, the automorphism group scheme $\Autsch(\phi_{\theta,\delta})$ is the twist
$\Res_{K/\Q}(\mu_7)/\mu_7$ of $\mu_7^3/\mu_7$, acting by $[\omega] \mapsto [\zeta\omega]$.

\begin{lemma}
  \label{lem:fermat-twist}
  Let $\delta \in K^{\times}$ and $\eta \colonequals \delta/f'(\theta)$.
  \begin{enumerate}
  \item \label{item:diagonal} Over $K$, the forms $\sfu_i$ are
    linear coordinates on $\PP^{2}$, in which $C_{\theta,\delta}$
    becomes the diagonal septic
    \begin{equation}
      \label{eq:fermat-twist}
      \eta\,\sfu_0^{7} + \sigma(\eta)\,\sfu_1^{7} + \sigma^{2}(\eta)\,\sfu_2^{7} = 0 .
    \end{equation}
    In particular $C_{\theta,\delta}$ is a twist of $\calF$, a smooth plane
    curve of genus $15$, and $\phi_{\theta,\delta}$ is the map of degree $49$ of
    \cite{Bruin06-nn2}*{Theorem 3}, branched over the three roots of $f$.
  \item \label{item:twist-class} The covering $\phi_{\theta,\delta}$ is the twist of $\phi_0$
    by the image of $[\eta]$ under the isomorphism
    \[
      K^{\times}/\Q^{\times}K^{\times 7} \cong H^1\bigl(\Q,\Aut(\phi_{0,\Qbar})\bigr),
    \]
    and every twist of $\phi_0$ arises in this way. In particular
    $(C_{\theta,\delta},\phi_{\theta,\delta})$ depends, up to isomorphism, only on
    the class of $\eta$ in $K^{\times}/\Q^{\times}K^{\times 7}$.
  \item \label{item:C0} Over $\Q$, $C_0$ is the twist of $\calF$ by the cocycle
    $\Gal(K/\Q) \to \Aut(\calF)$ sending $\sigma$ to the cyclic permutation of
    the coordinates.
  \end{enumerate}
\end{lemma}
\begin{proof}
  \eqref{item:diagonal} By \cref{lem:trace-form}, $P^{(2)}_{\theta,\delta} =
  \Tr(\eta\,\sfu_0^{7}) = \sum_{i} \sigma^{i}(\eta\,\sfu_0^{7}) = \sum_{i}
  \sigma^{i}(\eta)\,\sfu_i^{7}$, which is \cref{eq:fermat-twist}. The
  forms are $\sfu_i = \sum_{n} \sigma^{i}(\theta)^{n}\,\sfa_n$, whose
  matrix $\bigl(\sigma^{i}(\theta)^{n}\bigr)_{i,n}$ is the Vandermonde matrix
  of the three roots of $f$. Since $f$ is monic,
  \[
    \det\bigl(\sigma^{i}(\theta)^{n}\bigr)_{i,n}^{2}
    = \prod_{i < i'} \bigl(\sigma^{i}(\theta) - \sigma^{i'}(\theta)\bigr)^{2}
    = \disc(f) \neq 0,
  \]
  so $(\sfa_0,\sfa_1,\sfa_2) \mapsto (\sfu_0,\sfu_1,\sfu_2)$ is a linear change of
  coordinates over $K$. A diagonal form whose three coefficients are nonzero
  defines a smooth plane curve, which over $\Qbar$ is isomorphic to $\calF$ by
  rescaling the coordinates, and a smooth plane curve of degree $7$ has genus
  $(7-1)(7-2)/2 = 15$. The last assertion is \cite{Bruin06-nn2}*{Theorem 3};
  the degree also follows from the description of
  $\Aut(\phi_{\theta,\delta,\Qbar})$ in \cref{eq:aut-phi}, since the covering is Galois.

  \eqref{item:twist-class} Identify $K \otimes_{\Q} \Qbar$ with $\Qbar^{3}$ via
  $a \otimes b \mapsto (\sigma^{i}(a)b)_i$, and choose $\xi =
  (\xi_0,\xi_1,\xi_2) \in (K \otimes_{\Q} \Qbar)^{\times}$ with
  $\xi^7 = \eta$, that is, $\xi_i^7 = \sigma^{i}(\eta)$. By
  \eqref{item:diagonal}, $(\sfu_0:\sfu_1:\sfu_2) \mapsto
  (\xi_0\sfu_0:\xi_1\sfu_1:\xi_2\sfu_2)$ is an isomorphism
  $C_{\theta,\delta} \to C_0$ over $\Qbar$. It commutes with the maps to
  $\Pone$, because it carries the conjugates $\sigma^{i}(\delta\,\sfu_0^7) =
  \sigma^{i}(f'(\theta))\,\sigma^{i}(\eta)\,\sfu_i^7$, which determine
  $P^{(0)}_{\theta,\delta}$ and $P^{(1)}_{\theta,\delta}$, to those of
  $f'(\theta)\,\sfu_0^7$. The associated cocycle $\tau \mapsto \tau(\xi)/\xi$
  is the Kummer cocycle of $\eta$, with values in $\Res_{K/\Q}(\mu_7)(\Qbar)$,
  pushed forward to $\Aut(\phi_{0,\Qbar}) = (\Res_{K/\Q}(\mu_7)/\mu_7)(\Qbar)$. It
  remains to compute $H^1(\Q,\Aut(\phi_{0,\Qbar}))$. The exact sequence $1 \to \mu_7 \to \Res_{K/\Q}(\mu_7) \to
  \Res_{K/\Q}(\mu_7)/\mu_7 \to 1$, together with Shapiro's lemma and Hilbert~90,
  gives the exact sequence
  \[
    \Q^{\times}/\Q^{\times 7} \to K^{\times}/K^{\times 7} \to
    H^1\bigl(\Q, \Aut(\phi_{0,\Qbar})\bigr) \to \Br(\Q)[7] \to \Br(K)[7] .
  \]
  The last map is injective, since its composite with the corestriction is
  multiplication by $[K:\Q] = 3$.
  Hence $H^1\bigl(\Q,\Aut(\phi_{0,\Qbar})\bigr) \cong
  K^{\times}/\Q^{\times}K^{\times 7}$, and every class is the class of some
  $\eta$, that is, of $\phi_{\theta,\delta}$ with $\delta = \eta f'(\theta)$.

  \eqref{item:C0} For $\eta = 1$, \cref{eq:fermat-twist} is the equation of
  $\calF$, and $\tau \in \Gal_{\Q}$ acts on the coordinates $\sfu_i$ through
  the cyclic permutation $\sfu_i \mapsto \sfu_{i+a}$, where $\tau|_K = \sigma^a$.
\end{proof}

\begin{lemma}
  \label{lem:one-family}
  Let $\Sigma$ be the image of $\OO_K[1/7]^{\times}$ in
  $K^{\times}/\Q^{\times}K^{\times 7}$.
  \begin{enumerate}
  \item \label{item:eta-class} The pair $(C_{\theta,\delta},\phi_{\theta,\delta})$
    depends only on the class of $\eta = \delta/f'(\theta)$ in
    $K^{\times}/\Q^{\times}K^{\times 7}$: replacing $\delta$ by $r\nu^{7}\delta$
    with $r \in \Q^{\times}$ and $\nu \in K^{\times}$ amounts to the change of
    variables $\omega \mapsto \nu\omega$.
  \item \label{item:bijection} $\Sigma$ has $49$ elements, and for each of the
    six pairs $(F,k)$ the map $\delta \mapsto [\delta/f'(\theta)]$ is a bijection
    $\Delta(\theta,k) \to \Sigma$.
  \item \label{item:same-curves} The six descents involve \emph{the same}
    family of plane septics $\{\Tr(\eta\omega^{7}) = 0\}$, indexed by
    $[\eta] \in \Sigma$, and only the
    maps to $\Pone$ differ: in Zywina's coordinates they are related by the
    M\"obius transformation $t_{\ns} = 5 - 2t_{\spc}$.
  \end{enumerate}
\end{lemma}
\begin{proof}
  \eqref{item:eta-class} Since $r\nu^{7}\delta\,\omega^{7} = r\,\delta\,(\nu\omega)^{7}$,
  the map $[\omega] \mapsto [\nu\omega]$ carries $C_{\theta,r\nu^{7}\delta}$ onto
  $C_{\theta,\delta}$ compatibly with the maps to $\Pone$, and the factor $r$
  changes neither the condition $P^{(2)}_{\theta,\delta} = 0$ nor the point
  $(P^{(0)}_{\theta,\delta} : -P^{(1)}_{\theta,\delta})$. (This is also
  \cref{lem:fermat-twist}\eqref{item:twist-class}.)

  \eqref{item:bijection} Since $2$ is inert in $K$,
  $\OKS^{\times} = 2^{\Z} \times \OO_K[1/7]^{\times}$ when $\calS = \{2,7\}$;
  and $f'(\theta) \in \Q^{\times}\OO_K[1/7]^{\times}$ for both roots, by
  \cref{eq:fprime}.
  So $\delta \mapsto [\delta/f'(\theta)]$ is a well-defined map $\Delta(\theta,k) \to \Sigma$.
  It is injective. Let $\delta, \delta' \in \Delta(\theta,k)$ with
  $\delta' = r\nu^{7}\delta$ for some $r \in \Q^{\times}$ and $\nu \in K^{\times}$.
  Taking norms gives $r^{3} \in \Q^{\times 7}$, since $\Norm(\delta)/k$ and
  $\Norm(\delta')/k$ are seventh powers; hence $r \in \Q^{\times 7}$, so
  $\delta' \in \delta K^{\times 7}$. A seventh root in $K$ of an $\calS$-unit
  is an $\calS$-unit, so $\delta' \in \delta\,\OKS^{\times 7}$, and
  $\delta = \delta'$ because both are among the fixed representatives of
  $K(7,\calS)$. It remains to count. Let $\Q_{\calS}^{\times}$ be the group of
  $\calS$-units of $\Q$. The norm induces a homomorphism $K(7,\calS) \to
  \Q_{\calS}^{\times}/\Q_{\calS}^{\times 7}$, whose target has order $7^{\#\calS}$,
  and $\Delta(\theta,k)$ is a set of representatives of the fiber above the class of $k$.
  This homomorphism is surjective, since $\Norm(2) = 2^{3}$ and
  $\Norm(1 + \theta_{\spc}) = 7$; so every fiber has $\#K(7,\calS)/7^{\#\calS} = 49$
  elements, by \cref{tab:descent-data}. Finally, $\#\Sigma = 49$ by
  \cref{lem:V}\eqref{item:F72} below, whose proof is independent of this
  section. An injective map between two sets of $49$ elements is a bijection.

  \eqref{item:same-curves} By \eqref{item:bijection}, the set of classes
  $[\eta]$, $\delta \in \Delta(\theta,k)$, is $\Sigma$ for all six pairs, and by
  \eqref{item:eta-class} the curves are determined by these classes. In the
  intrinsic description of \cref{rem:intrinsic}, $C_{\theta,\delta} = \{[\omega]
  : \delta\omega^{7} \in \Q + \Q\theta\}$, and $\Q + \Q\theta_{\ns} = \Q +
  \Q\theta_{\spc}$; so the curves are the same, and only the identification of
  $\PP(\Q + \Q\theta)$ with $\Pone_{\t}$ differs. Explicitly, $x - \theta_{\ns}y =
  (x - 5y) - \theta_{\spc}(-2y)$, so $(x : y)_{\ns} = (x - 5y : -2y)_{\spc}$,
  that is, $t_{\ns} = 5 - 2t_{\spc}$.
\end{proof}

\section{The Klein quartic quotient}
\label{sec:klein}

Let $\calZ \colon \sfv_0^3\sfv_1 + \sfv_1^3\sfv_2 + \sfv_2^3\sfv_0 = 0$ be the
Klein quartic. Following Elkies \cite{Elkies-99}*{Section 3.2, (3.5)--(3.6)},
we consider
\begin{equation}
  \label{eq:psi}
  \psi \colon \calF \to \calZ, \qquad (\sfv_0:\sfv_1:\sfv_2) \colonequals (\sfu_0^3\sfu_2
  : \sfu_1^3\sfu_0 : \sfu_2^3\sfu_1),
\end{equation}
which is well defined because of the identity
\begin{equation}
  \label{eq:elkies-identity}
  \sfv_0^3\sfv_1 + \sfv_1^3\sfv_2 + \sfv_2^3\sfv_0
  = (\sfu_0\sfu_1\sfu_2)^{3}\,\bigl(\sfu_0^7 + \sfu_1^7 + \sfu_2^7\bigr).
\end{equation}
The map $\psi$ has two properties that we use.
\begin{enumerate}
\item It is the quotient of $\calF$ by the group $H$ of order $7$ generated by
  the geometric automorphism
  $(\sfu_0:\sfu_1:\sfu_2) \mapsto (\zeta_7\sfu_0 : \zeta_7^{2}\sfu_1 :
  \zeta_7^{4}\sfu_2)$. Indeed, $\sfu_0^3\sfu_2$, $\sfu_1^3\sfu_0$,
  $\sfu_2^3\sfu_1$ are the only monomials of degree $4$ invariant under this
  automorphism, so $\psi$ is constant on the orbits of $H$ and factors through
  $\calF/H$; and $\psi$ has degree $7 = \#H$ \cite{Elkies-99}*{Section 3.2, p.~74},
  so the induced map $\calF/H \to \calZ$ is birational, hence an isomorphism,
  $\calZ$ being smooth.
\item It commutes with the cyclic permutations of the coordinates
  $(\sfu_0,\sfu_1,\sfu_2) \mapsto (\sfu_1,\sfu_2,\sfu_0)$ and
  $(\sfv_0,\sfv_1,\sfv_2) \mapsto (\sfv_1,\sfv_2,\sfv_0)$.
\end{enumerate}
The second property is what allows $\psi$ to be twisted. By
\cref{lem:fermat-twist}, over $K$ and in the coordinates $\sfu_i$, the curve $C_{\theta,\delta}$ is a diagonal septic on which
$\sigma$ acts by cyclic permutation of the coordinates. So $\psi$ should induce
a map from $C_{\theta,\delta}$ to a twist of $\calZ$, defined over $\Q$. We now
make this precise.

\subsection{Twisting the quotient map}
Take the coordinates $\sfu_i$ of \cref{lem:fermat-twist} on the source, and write
$\sfv_i$ for the same linear forms, regarded as coordinates on the target. The
formula for $\psi$ reads $\sfv_i = \sfu_i^3\sfu_{i-1} = \sigma^i(\sfu_0^3\sfu_2)$,
where $\sfu_0^3\sfu_2 = \sfu_0^3\,\sigma^{-1}(\sfu_0)$. This expression is Galois
equivariant, so we obtain a rational map defined over $\Q$, which we denote by the
same letter,
\[
  \psi \colon \PP^2_\Q \dashrightarrow \PP^2_\Q, \qquad [\omega] \mapsto [\omega^{3}\sigma^{-1}(\omega)];
\]
it is given by three quartic forms with rational coefficients, and over $K$ it
restricts to the map $\psi \colon \calF \to \calZ$ of \cref{eq:psi} on $C_0 \times_\Q K \cong
\calF$. For $\delta \in K^{\times}$, define quartic forms $Q^{(n)}_{\theta,\delta}
\in \Q[\sfa_0,\sfa_1,\sfa_2]$ by
\begin{equation}
  \label{eq:Q-forms}
  \delta\, \sfu_0^{3}\sfu_1
  = Q^{(0)}_{\theta,\delta} + Q^{(1)}_{\theta,\delta}\,\theta + Q^{(2)}_{\theta,\delta}\,\theta^{2},
\end{equation}
exactly as the forms $P^{(n)}_{\theta,\delta}$ were defined from $\sfu_0^7$ in
\cref{eq:P-forms}, and put
\begin{equation}
  \label{eq:Z-and-varphi}
  Z_{\theta,\delta} \colonequals \bigl\{Q^{(2)}_{\theta,\delta} = 0\bigr\} \subset \PP^2_{\Q},
  \qquad
  \varphi_{\theta,\delta} \colon Z_{\theta,\delta} \dashrightarrow \Pone_{\Q},
  \quad [v] \mapsto \bigl(Q^{(0)}_{\theta,\delta} : -Q^{(1)}_{\theta,\delta}\bigr).
\end{equation}
As in \cref{lem:trace-form}, $Q^{(2)}_{\theta,\delta} = \Tr\bigl(\eta\,
\sfu_0^3\sfu_1\bigr)$ with $\eta = \delta/f'(\theta)$; so
$Z_{\theta,\delta} = \{\Tr(\eta\,v^3\sigma(v)) = 0\}$ in the notation of
\cref{sec:descent}, and over $K$, in the coordinates $\sfv_i$, it becomes
\begin{equation}
  \label{eq:klein-twist}
  \eta\,\sfv_0^{3}\sfv_1 + \sigma(\eta)\,\sfv_1^{3}\sfv_2 + \sigma^{2}(\eta)\,\sfv_2^{3}\sfv_0 = 0,
\end{equation}
a twist of the Klein quartic $\calZ$.

\begin{lemma}
  \label{lem:quotient-map}
  For every $\delta \in K^{\times}$ and $n \in \{0,1,2\}$,
  \[
    Q^{(n)}_{\theta,\delta} \circ \psi
    = (\sfu_0\sfu_1\sfu_2)^{3}\, P^{(n)}_{\theta,\delta} .
  \]
  Consequently $\psi$ maps $C_{\theta,\delta}$ to $Z_{\theta,\delta}$, and the
  diagram
  \[
    \begin{tikzcd}
      \calF \arrow[r, "\psi"] & \calZ \\
      C_{\theta,\delta} \arrow[r, "\psi"] \arrow[u, "\cong"]
        \arrow[dr, "\phi_{\theta,\delta}"'] &
      Z_{\theta,\delta} \arrow[u, "\cong"'] \arrow[d, "\varphi_{\theta,\delta}"] \\
      & \Pone_{\Q}
    \end{tikzcd}
  \]
  commutes: the triangle over $\Q$, that is, $\phi_{\theta,\delta} =
  \varphi_{\theta,\delta} \circ \psi$, and the square over $\Qbar$, where the
  vertical isomorphisms are $\sfu_i \mapsto \xi_i\sfu_i$ and $\sfv_i
  \mapsto \xi_i^{3}\xi_{i-1}\sfv_i$ (indices modulo $3$) in the
  coordinates of \cref{eq:fermat-twist} and \cref{eq:klein-twist}, for any
  $\xi_i \in \Qbar$ with $\xi_i^7 = \sigma^i(\eta)$. On rational
  points, $\psi$ and $\varphi_{\theta,\delta}$ are everywhere defined.
\end{lemma}
\begin{proof}
  Substituting $\psi$ replaces $\sfu_i$ by $\sfu_i^3\sfu_{i-1}$, so it sends
  $\sfu_0^3\sfu_1$ to $(\sfu_0^3\sfu_2)^3\,\sfu_1^3\sfu_0 =
  (\sfu_0\sfu_1\sfu_2)^3\,\sfu_0^7$; on points, $v = \omega^3\sigma^{-1}(\omega)$
  satisfies $v^3\sigma(v) = \Norm(\omega)^3\omega^7$. In the group ring $\Z[\Gal(K/\Q)]$, this is the
  identity $(3 + \sigma)(3 + \sigma^{-1}) = 7 + 3(1 + \sigma + \sigma^{2})$.
  Multiply by $\delta$ and compare \cref{eq:P-forms} with \cref{eq:Q-forms}; this
  gives the displayed identity, hence the triangle. For the square, put
  $\sfu_i' = \xi_i\sfu_i$ and $\sfv_i' = \sfu_i'^{3}\sfu_{i-1}' =
  \xi_i^3\xi_{i-1}\sfv_i$. Then $\sum_i \sfu_i'^{7} = \sum_i
  \sigma^i(\eta)\,\sfu_i^7$ and
  \[
    \sum_i \sfv_i'^{3}\sfv_{i+1}'
    = \sum_i \xi_i^{10}\xi_{i-1}^{3}\xi_{i+1}^{3}\,\sfv_i^3\sfv_{i+1}
    = (\xi_0\xi_1\xi_2)^{3} \sum_i \sigma^i(\eta)\,\sfv_i^3\sfv_{i+1},
  \]
  so the vertical maps are isomorphisms onto $\calF$ and $\calZ$, and they
  intertwine the two maps $\psi$. A rational point is the class of some $\omega
  \in K^{\times}$, and then $v = \omega^3\sigma^{-1}(\omega) \neq 0$ and $\delta
  v^3\sigma(v) \neq 0$.
\end{proof}

\begin{remark}
  \label{rem:elkies-map}
  The map $\psi \colon \calF \to \calZ$ is a cyclic unramified cover of degree
  $7$; indeed $2 \cdot 15 - 2 = 7\,(2 \cdot 3 - 2)$. Its deck group, generated
  by $(\sfu_0:\sfu_1:\sfu_2) \mapsto (\zeta_7\sfu_0 : \zeta_7^{2}\sfu_1 :
  \zeta_7^{4}\sfu_2)$, is one of the two subgroups of order $7$ of
  $\Aut(\phi_{\theta,\delta,\Qbar}) = \mu_7(\Qbar)^3/\mu_7(\Qbar)$ that are
  stable under the cyclic permutation of the coordinates, which is why $\psi$
  descends to $\Q$. The other one, generated by $(\sfu_0:\sfu_1:\sfu_2) \mapsto
  (\zeta_7\sfu_0 : \zeta_7^{4}\sfu_1 : \zeta_7^{2}\sfu_2)$, has the invariant
  monomials $\sfu_0^3\sfu_1$, $\sfu_1^3\sfu_2$, $\sfu_2^3\sfu_0$; it gives
  $[\omega] \mapsto [\omega^3\sigma(\omega)]$ and a second twist of the Klein
  quartic, $\{\Tr(\eta\, v^3\sigma^{-1}(v)) = 0\}$. For $\eta = 1$ this second curve is the twist of
  $\calZ$ by the cocycle $\Gal(K/\Q) \to \Aut(\calZ)$ sending
  $\sigma$ to the inverse of the cyclic permutation, and it is not trivial: it
  has no points over $\Q_2$, whereas $Z_0 \cong \calZ$ by \cref{cor:Z0}. The choice
  of $\sigma$ is therefore not a convention. Twists
  of the Klein quartic failing the Hasse principle are studied in
  \cite{LGV-24}; see \cite{FLS-18} for the classification of the twists. The other
  six subgroups of order $7$ of $\Aut(\phi_{\theta,\delta,\Qbar})$ form two
  orbits of three under the cyclic permutation of the coordinates, so the
  corresponding quotient maps of $C_{\theta,\delta}$ are defined over $K$ but
  not over $\Q$. One orbit consists of the three subgroups generated by
  $(\sfu_0:\sfu_1:\sfu_2) \mapsto (\zeta_7\sfu_0:\sfu_1:\sfu_2)$ and its
  permutations, whose quotient maps are the projections
  $(\sfu_0:\sfu_1:\sfu_2) \mapsto (\sfu_1:\sfu_2)$ onto $\Pone$; the other
  three subgroups act freely, and their quotients are hyperelliptic curves of
  genus $3$. The coverings $\psi$ are twisted forms of the covering of the Klein
  quartic attached to the isogeny $1 - \zeta_7$ of its Jacobian, as in the
  descent theory for cyclic covers of Poonen and Schaefer \cite{PS-97}; compare
  \cite{PSS-07}*{Section 10.1}.
\end{remark}

\begin{lemma}
  \label{lem:twisting}
  Let $\delta,\delta' \in K^{\times}$ with $\delta' = r\,\nu^{3}\sigma(\nu)\,\delta$
  for some $r \in \Q^{\times}$ and $\nu \in K^{\times}$. Then $[v] \mapsto [\nu v]$
  is an isomorphism $Z_{\theta,\delta'} \to Z_{\theta,\delta}$ over $\Q$, and it
  commutes with the maps $\varphi$. Moreover, for every $\lambda \in K^{\times}$,
  \begin{equation}
    \label{eq:seventh-powers}
    \lambda^{7} = \Norm(\lambda)\cdot \nu^{3}\sigma(\nu),
    \qquad \nu = \lambda^{2}/\sigma(\lambda),
  \end{equation}
  so that $Z_{\theta,\delta}$ and $\varphi_{\theta,\delta}$ only depend on the
  class of $\delta/f'(\theta)$ in the quotient $K^{\times}/\Q^{\times}(K^{\times})^{3+\sigma}$ of
  $K^{\times}/\Q^{\times}K^{\times 7}$, where $(K^{\times})^{3+\sigma} \colonequals
  \{\nu^{3}\sigma(\nu) : \nu \in K^{\times}\}$.
\end{lemma}
\begin{proof}
  The first statement follows from $\delta'v^{3}\sigma(v) = r\,\delta(\nu
  v)^{3}\sigma(\nu v)$. For the second, $\nu^{3}\sigma(\nu) =
  \lambda^{6}\sigma(\lambda)^{-3}\sigma(\lambda)^{2}\sigma^{2}(\lambda)^{-1} = \lambda^{7}/\Norm(\lambda)$.
\end{proof}

\subsection{The untwisted quartic \texorpdfstring{$Z_0$}{Z0}}
We use the following identity of Klein \cite{Klein-1879}*{\S 5}; see also
\cite{Elkies-99}*{(1.3)}.
\begin{proposition}[Klein's identity]
  \label{prop:klein-identity}
  Let $\sqrt{-7} \colonequals \zeta_7 + \zeta_7^2 + \zeta_7^4 - \zeta_7^3 -
  \zeta_7^5 - \zeta_7^6$, and let
  \begin{equation}
    \label{eq:alpha}
    \alpha \colonequals \sqrt{-7}\,\bigl(\zeta_7^{4} - \zeta_7^{-4}\bigr)
    = 4c_1 + 2c_2 + c_4
    = \theta_{\spc}^{2} - 5\theta_{\spc} + 1 \in \OO_K ,
    \qquad c_a \colonequals \zeta_7^{a} + \zeta_7^{-a}.
  \end{equation}
  For $\vartheta = \tfrac{1}{7}\bigl(\sfv_0\,\alpha + \sfv_1\,\sigma(\alpha) +
  \sfv_2\,\sigma^{2}(\alpha)\bigr) \in K[\sfv_0,\sfv_1,\sfv_2]$, with $\sigma$
  acting on the coefficients, we have
  \[
    \Tr\bigl(\vartheta^{3}\sigma(\vartheta)\bigr) = \sfv_0^{3}\sfv_1 + \sfv_1^{3}\sfv_2 + \sfv_2^{3}\sfv_0 .
  \]
\end{proposition}
\begin{proof}
  We have $\sigma(\alpha) = \sqrt{-7}(\zeta_7 - \zeta_7^{-1})$ and
  $\sigma^2(\alpha) = \sqrt{-7}(\zeta_7^2 - \zeta_7^{-2})$. Since $\sqrt{-7}/7 =
  -1/\sqrt{-7}$, the numbers $\alpha_i \colonequals \sigma^i(\alpha)/7$ are the
  entries $-(\zeta_7^{a} - \zeta_7^{-a})/\sqrt{-7}$, $a \in \{4,1,2\}$, of the
  symmetric matrix of order two in Klein's group of $168$ automorphisms of his
  quartic.

  Let $\Phi(\sfv_0,\sfv_1,\sfv_2) = \sfv_0^3\sfv_1 + \sfv_1^3\sfv_2 +
  \sfv_2^3\sfv_0$. Klein's matrix
  \[
    M \colonequals \begin{pmatrix*}[r] \alpha_1 & \alpha_2 & \alpha_0 \\ \alpha_2 & \alpha_0 & \alpha_1 \\ \alpha_0 & \alpha_1 & \alpha_2 \end{pmatrix*}
    \in \SL_3(K)
  \]
  satisfies $M^2 = I$ and $\Phi \circ M = \Phi$. Each row of $M$ is the image under
  $\sigma$ of the previous one. So, for the column vector
  $\mathbf{s} = (s_1,s_2,s_3) \colonequals (\sfv_1,\sfv_2,\sfv_0)$, whose entries are fixed by $\sigma$, the entries of
  $M\mathbf{s}$ are $\vartheta, \sigma(\vartheta), \sigma^{2}(\vartheta)$, since $\alpha_1 s_1 + \alpha_2 s_2 + \alpha_0 s_3 = \vartheta$. Hence
  $\Phi(\mathbf{s}) = \Phi(M\mathbf{s}) = \Tr(\vartheta^{3}\sigma(\vartheta))$, and
  $\Phi(\mathbf{s}) = \Phi(\sfv_0,\sfv_1,\sfv_2)$ because $\Phi$ is invariant under cyclic
  permutations. (The identity can also be checked by expanding both sides.)
\end{proof}

\begin{lemma}[Hurwitz]
  \label{lem:hurwitz}
  The rational points of the Klein quartic are the trivial ones:
  \[\calZ(\Q) = \{(1:0:0),\ (0:1:0),\ (0:0:1)\}.\]
\end{lemma}
\begin{proof}
  We summarize the argument of Hurwitz \cite{Hurwitz-08}, which reduces the
  statement to Fermat's Last Theorem for exponent $7$, proved by Lam\'e
  \cite{Lame-1840}. Let $(x:y:z) \in \calZ(\Q)$ with $x,y,z \in \Z$ coprime. If one
  coordinate vanishes, say $z = 0$, then $x^3y = 0$, and the point is one of
  the three above. Suppose that $xyz \neq 0$, and let $p$ be a prime. The three
  terms $x^3y$, $y^3z$, $z^3x$ sum to zero, so the smallest of their $p$-adic
  valuations is attained at least twice. If $p$ divided exactly one of $x,y,z$,
  say $x$, then $y^3z$ would be the only term prime to $p$; so every prime
  dividing $xyz$ divides exactly two of $x,y,z$. If $p \mid x$ and $p \mid y$, then
  $v_p(z^3x) = v_p(x) < v_p(x^3y)$ forces $v_p(x) = 3v_p(y)$, and cyclically.
  Hence $|x| = a^3c$, $|y| = b^3a$, and $|z| = c^3b$ for
  \begin{align*}
    a = \prod_{p \mid \gcd(x,y)} p^{v_p(y)},\quad b = \prod_{p \mid \gcd(y,z)} p^{v_p(z)},\quad c = \prod_{p \mid \gcd(z,x)} p^{v_p(x)}.
  \end{align*}
  Replacing $(x,y,z)$ by $(-x,-y,-z)$ if necessary, we may assume that $xyz > 0$,
  and then signs can be given to $a,b,c$ so that $(x,y,z) = (a^3c, b^3a,
  c^3b)$. Now \cref{eq:elkies-identity} gives $(abc)^3(a^7 + b^7 + c^7) = 0$,
  with $abc \neq 0$, which contradicts Lam\'e's theorem.
\end{proof}

\begin{remark}
  \label{rem:elkies-hurwitz}
  In the language of this section, the proof lifts a rational point of $\calZ$
  to a rational point of $\calF$ along the cyclic unramified cover
  $(\sfu_0:\sfu_1:\sfu_2) \mapsto (\sfu_0^3\sfu_2 : \sfu_1^3\sfu_0 :
  \sfu_2^3\sfu_1)$. Elkies \cite{Elkies-99}*{Section 3.1} gives a proof that
  does not use Lam\'e's theorem: the quotient of $\calZ$ by the cyclic
  permutation of the coordinates is an elliptic curve of conductor $49$ whose
  Mordell--Weil group has order $2$, as a $2$-descent in the style of Fermat
  shows, and the fibers above its two rational points contain no rational
  points other than the three points of \cref{lem:hurwitz}. Read in the opposite direction, through
  the same cover, this gives a proof of Lam\'e's theorem
  \cite{Elkies-99}*{Section 3.2}.
\end{remark}

\begin{corollary}
  \label{cor:Z0}
  The \cdef{untwisted quartic}
  $Z_0 \colonequals \{\Tr(v^{3}\sigma(v)) = 0\} = Z_{\theta,f'(\theta)}$
  is isomorphic over $\Q$ to the Klein quartic, and
  \[
    Z_0(\Q) = \bigl\{[\alpha],\ [\sigma(\alpha)],\ [\sigma^{2}(\alpha)]\bigr\}.
  \]
\end{corollary}
\begin{proof}
  By \cref{prop:klein-identity}, $Z_0 \cong \calZ$, with the three points
  above corresponding to $(1:0:0)$, $(0:1:0)$, $(0:0:1)$. These are the only
  rational points of $\calZ$, by \cref{lem:hurwitz}.
\end{proof}

\section{Proof of \texorpdfstring{\cref{thm:superelliptic}}{Theorem 1}}
\label{sec:proof-thm-1}

In this section (with the exception of \cref{prop:seven-adic})
$\theta = \theta_{\spc}$ and $f = f_{\spc}$, and $\varpi = 1 + \theta$ is as in
\cref{tab:descent-data}. Then
\begin{equation}
  \label{eq:varpi}
  \Norm(\varpi) = 7, \qquad 7 = \varpi^{3}\theta^{-2}, \qquad
  \sigma(\varpi) = \varpi(4 - \varpi), \qquad
  f'(\theta) = \varpi^{2}\Bigl(3 - \frac{2\varpi}{\theta}\Bigr),
\end{equation}
and $\varpi$ is a uniformizer of the completion $K_7$ of $K$ at its prime above
$7$, which is a totally ramified cubic extension of $\Q_7$ with residue field
$\F_7$. The element $\theta = \varpi - 1$ is a unit.

\subsection{A homomorphism}
Recall from \cref{lem:one-family} the group $\Sigma$, the image of
$\OO_K[1/7]^{\times}$ in $K^{\times}/\Q^{\times}K^{\times 7}$. By
\cref{lem:twisting}, the quartic $Z_{\theta,\delta}$ is isomorphic to $Z_0$
over $\Q$, compatibly with the maps to $\Pone$, as soon as the class of
$\eta = \delta/f'(\theta)$ lies in $(3+\sigma)\Sigma$. We shall see in
\cref{lem:V} that $\Sigma \cong \F_7^{2}$ is the sum of the eigenspaces of
$\sigma$ for the eigenvalues $2$ and $4$, and that $(3+\sigma)\Sigma$ is the
$2$-eigenspace, because $3 + 4 = 0$ in $\F_7$. What we need is therefore a
homomorphism $\Sigma \to \F_7$ that kills the $2$-eigenspace and not the
$4$-eigenspace, and whose value on $(x - \theta y)/f'(\theta)$ we can compute.
A nonzero homomorphism $\phi$ with $\phi \circ \sigma = 4\phi$ has the first
property, since $2\phi(v) = \phi(\sigma v) = 4\phi(v)$ forces $\phi(v) = 0$ on
the $2$-eigenspace. The following $7$-adic map $\kappa$ has both properties.

Every $g \in K_7^{\times}$ can be written as $g = \varpi^{a}u$ with $a \in \Z$ and
$u \equiv u_0(1 + b\varpi) \pmod{\varpi^2}$, where $u_0 \in \F_7^{\times}$ and
$b \in \F_7$. We define
\[
  \kappa \colon K_7^{\times} \to \F_7, \qquad \kappa(g) \colonequals b + 4a .
\]

\begin{lemma}
  \label{lem:kappa}
  The map $\kappa$ is a homomorphism. It is trivial on $\Q_7^{\times}$ and on
  seventh powers, and $\kappa \circ \sigma = 4\kappa$.
\end{lemma}
\begin{proof}
  Both $a$ and $b$ are additive, and seventh powers die in $\F_7$. A unit of
  $\Z_7$ is congruent to an element of $\F_7^\times$ modulo $7$, hence modulo
  $\varpi^3$, so $\kappa(\Z_7^{\times}) = 0$. By \cref{eq:varpi}, $7 =
  \varpi^{3}\theta^{-2}$ with $\theta^{-2} = (1-\varpi)^{-2} \equiv 1 + 2\varpi
  \pmod{\varpi^{2}}$, so $\kappa(7) = 2 + 12 = 0$. Finally, $\sigma$ acts
  trivially on the residue field, and $\sigma(\varpi) = \varpi(4 - \varpi) \equiv
  4\varpi(1 - 2\varpi) \pmod{\varpi^{3}}$, since $4^{-1} = 2$ in $\F_7$. Hence, for
  $g = \varpi^a u$ as above, $\sigma(g) \equiv \varpi^{a}\cdot 4^{a}(1 -
  2\varpi)^{a}\sigma(u)$ modulo $\varpi^{a+2}$, with $\sigma(u) \equiv u_0(1 +
  4b\varpi) \pmod{\varpi^{2}}$, and $\kappa(\sigma(g)) = (4b - 2a) + 4a = 4b + 2a
  = 4\kappa(g)$.
\end{proof}

The group $\Sigma$ is an $\F_7[\sigma]$-module, and $1 + \sigma + \sigma^2$ acts
as zero on it, because norms are rational. As $\t^{2} + \t + 1 = (\t - 2)(\t -
4)$ in $\F_7[\t]$, the action of $\sigma$ on $\Sigma$ is diagonalizable, with
eigenvalues among the primitive cube roots of unity $2$ and $4$ of $\F_7$;
\cref{lem:V} shows that both occur.

\begin{lemma}
  \label{lem:V}
  Let $\varepsilon_1 = \theta - 1$ and $\varepsilon_2 = 2 - \theta$ be the units of
  \cref{tab:descent-data}. Then the following hold.
  \begin{enumerate}
  \item \label{item:units} $\OO_K^{\times} = \pm\langle \varepsilon_1,\varepsilon_2\rangle$,
    with $\sigma(\varepsilon_1) = \varepsilon_2^{-1}$ and
    $\sigma(\varepsilon_2) = -\varepsilon_1\varepsilon_2^{-1}$.
  \item \label{item:F72} $\Sigma$ is generated by the classes of $\varepsilon_1$
    and $\varepsilon_2$, and $\OO_K^{\times} \to \Sigma$ induces an isomorphism
    $\OO_K^{\times}/\pm\OO_K^{\times 7} \cong \Sigma$; in particular
    $\Sigma \cong \F_7^{2}$.
  \item \label{item:V-structure} $\Sigma = \Sigma_2 \oplus \Sigma_4$, where
    $\Sigma_2$ is generated by $[\varepsilon_1\varepsilon_2^{2}]$ and $\sigma$
    acts on it as multiplication by $2$, and $\Sigma_4$ is generated by
    $[\varepsilon_1\varepsilon_2^{4}]$ and $\sigma$ acts on it as multiplication
    by $4$.
  \item \label{item:kernel} $\kappa$ induces a homomorphism $\Sigma \to \F_7$,
    with $\kappa(\varepsilon_1) = 3$ and $\kappa(\varepsilon_2) = 2$, and
    \[
      \ker\bigl(\kappa|_{\Sigma}\bigr) = \Sigma_2 = (3+\sigma)\Sigma .
    \]
  \end{enumerate}
\end{lemma}
\begin{proof}
  In part \eqref{item:units}, the two relations are a direct computation, and
  that $-1$, $\varepsilon_1$ and $\varepsilon_2$ generate $\OO_K^{\times}$ is
  checked with \Magma{}. Since $\varpi^3 = 7\theta^2$,
  the class of $\varpi$ in $K^{\times}/\Q^{\times}K^{\times 7}$ equals that of
  $\theta^{-4}$, and $\theta$ is a unit; so $\Sigma$ is generated by the classes
  of units, hence by those of $\varepsilon_1$ and $\varepsilon_2$. If a unit $u$
  lies in $\Q^{\times}K^{\times 7}$, say $u = r\nu^{7}$ with $r \in \Q^{\times}$
  and $\nu \in K^{\times}$, then $\pm 1 = \Norm(u) = r^{3}\Norm(\nu)^{7}$, so
  $r \in \pm\Q^{\times 7}$, and $u = \pm(s\nu)^{7}$ with $s \in \Q^{\times}$, where
  $s\nu$ is a unit because its seventh power $\pm u$ is a unit. Hence the kernel of
  $\OO_K^{\times} \to \Sigma$ is $\pm\OO_K^{\times 7}$, and \eqref{item:F72}
  follows from \eqref{item:units}. For \eqref{item:V-structure}, in additive
  notation, $\sigma[\varepsilon_1] = -[\varepsilon_2]$ and $\sigma[\varepsilon_2]
  = [\varepsilon_1] - [\varepsilon_2]$, so
  $\sigma[\varepsilon_1\varepsilon_2^{2}] = 2[\varepsilon_1\varepsilon_2^{2}]$ and
  $\sigma[\varepsilon_1\varepsilon_2^{4}] = 4[\varepsilon_1\varepsilon_2^{4}]$, and
  these two classes span $\Sigma \cong \F_7^{2}$. By \cref{lem:kappa}, $\kappa$
  is well defined on $\Sigma$. From $\varepsilon_1 = \varpi - 2 \equiv -2(1 -
  4\varpi)$ and $\varepsilon_2 = 3 - \varpi \equiv 3(1 - 5\varpi)$ modulo
  $\varpi^{2}$ we get $\kappa(\varepsilon_1)
  = 3$ and $\kappa(\varepsilon_2) = 2$, so $\kappa(\varepsilon_1\varepsilon_2^{2})
  = 0$ and $\kappa(\varepsilon_1\varepsilon_2^{4}) = 4$. Finally $3 + \sigma$ acts
  as $5$ on $\Sigma_2$ and as $0$ on $\Sigma_4$.
\end{proof}

\subsection{The \texorpdfstring{$7$}{7}-adic step}
\label{ss:7-adic}
\begin{proposition}
  \label{prop:seven-adic}
  Let $\bullet \in \{\ns,\spc\}$, and write $f = f_{\bullet}$,
  $F = F_{\bullet}$ and $\theta = \theta_{\bullet}$. Let $x,y$ be coprime integers with
  $7 \mid F(x,y)$.
  \begin{enumerate}
  \item \label{item:kappa-zero} We have
    \[
      \kappa\left(\frac{x - \theta y}{f'(\theta)}\right) = 0 .
    \]
  \item \label{item:in-image} If moreover $v_p(F(x,y)) \equiv 0 \pmod 7$ for
    every prime $p \notin \{2,7\}$, then
    \[
      \frac{x - \theta y}{f'(\theta)} \in \Q^{\times}\cdot\bigl\{\nu^{3}\sigma(\nu) : \nu \in K^{\times}\bigr\}.
    \]
  \item \label{item:Z-is-Klein} Suppose that $F(x,y) = kw^{7}$ with $w \in \Z$,
    for one of the pairs $(F,k)$ of \cref{tab:descent-data} with $7 \mid k$ (so
    that $7 \mid F(x,y)$), and
    let $\delta \in \Delta(\theta,k)$ and $\omega \in K^{\times}$ be such that
    $x - \theta y = \delta\omega^{7}$, as in \cref{prop:bruin}. Then
    $\kappa(\delta/f'(\theta)) = 0$, and $\delta/f'(\theta) =
    r\,\nu^{3}\sigma(\nu)$ for some $r \in \Q^{\times}$ and $\nu \in
    K^{\times}$. In particular, $[v] \mapsto [\nu v]$ is an isomorphism from
    $Z_{\theta,\delta}$ to $Z_0$ over $\Q$, compatible with the maps to $\Pone$,
    and $Z_{\theta,\delta}$ is isomorphic to the Klein quartic $\calZ$ over
    $\Q$.
  \end{enumerate}
\end{proposition}
\begin{proof}
  We prove \eqref{item:kappa-zero}. Suppose first that $\bullet = \spc$. By
  \cref{lem:arithmetic-of-forms}\eqref{item:v7-sp}, $x + y = 7s$ for some integer
  $s$, and $7 \nmid y$. By \cref{eq:varpi},
  \[
    x - \theta y = -y\varpi + 7s = \varpi\bigl(-y + s\theta^{-2}\varpi^{2}\bigr),
  \]
  so $\kappa(x - \theta y) = 0 + 4 = 4$. On the other hand $f'(\theta) =
  \varpi^{2}(3 - 2\varpi/\theta)$ by \cref{eq:varpi}, and $3 - 2\varpi/\theta
  \equiv 3 + 2\varpi \equiv 3(1 + 3\varpi) \pmod{\varpi^{2}}$, since $\theta \equiv -1$
  and $2 \cdot 3^{-1} = 3$ in $\F_7$, so $\kappa(f'(\theta))
  = 3 + 8 = 4$ as well. If $\bullet = \ns$, then $7 \mid x$ and $7 \nmid y$, and by
  \cref{eq:thetas} and \cref{eq:fprime},
  \[
    \frac{x - \theta_{\ns}y}{f_{\ns}'(\theta_{\ns})}
    = \frac{1}{4}\cdot\frac{x' - \theta_{\spc}y'}{f_{\spc}'(\theta_{\spc})},
    \qquad (x',y') = (x - 5y, -2y),
  \]
  with $7 \mid x' + y'$ and $7 \nmid y'$; the computation above applies to
  $(x',y')$, and $\kappa(4) = 0$.

  For \eqref{item:in-image}, let $\gamma$ be the element in question. The
  argument in the proof of \cref{prop:bruin} shows that $x - \theta y \in
  \OKS^{\times}K^{\times 7}$ with $\calS = \{2,7\}$. Since $2$ is inert,
  $\OKS^{\times} = 2^{\Z} \times \OO_K[1/7]^{\times}$, and $f'(\theta) \in
  \Q^{\times}\OO_K[1/7]^{\times}$ by \cref{eq:fprime}; hence the class of
  $\gamma$ in $K^{\times}/\Q^{\times}K^{\times 7}$ lies in $\Sigma$. By
  \cref{lem:V}\eqref{item:kernel} this class lies in $(3+\sigma)\Sigma$, and by
  \cref{lem:V}\eqref{item:F72} it is therefore the class of $u^{3}\sigma(u)$ for
  a unit $u$. This means that $\gamma = r\,u^{3}\sigma(u)\lambda^{7}$ with $r \in
  \Q^{\times}$, $u \in \OO_K^{\times}$ and $\lambda \in K^{\times}$. We conclude with
  \cref{eq:seventh-powers}.

  For \eqref{item:Z-is-Klein}, we have $7 \mid F(x,y)$ because $7 \mid k$, and
  $v_p(F(x,y)) = 7v_p(w)$ for every prime $p \notin \{2,7\}$ because $k$
  divides $56$. Since $\kappa$ is trivial on seventh powers (\cref{lem:kappa}),
  \eqref{item:kappa-zero} gives
  \[
    \kappa\left(\frac{\delta}{f'(\theta)}\right)
    = \kappa\left(\frac{x - \theta y}{f'(\theta)}\right) - 7\kappa(\omega)
    = 0 .
  \]
  In other words, the $7$-adic condition depends only on the class of $x -
  \theta y$ in $K^{\times}/K^{\times 7}$, which is the class of $\delta$; so it
  is a condition on $\delta \in \Delta(\theta,k)$. By \eqref{item:in-image} and
  \cref{eq:seventh-powers}, $\delta/f'(\theta) = \bigl((x-\theta y)/f'(\theta)\bigr)
  \omega^{-7}$ lies in $\Q^{\times}\cdot\{\nu^{3}\sigma(\nu) : \nu \in K^{\times}\}$.
  Now \cref{lem:twisting}, applied to $\delta$ and $f'(\theta)$ in place of
  $\delta'$ and $\delta$, gives the isomorphism $Z_{\theta,\delta} \to
  Z_{\theta,f'(\theta)} = Z_0$, and \cref{cor:Z0} gives $Z_0 \cong \calZ$.
\end{proof}

\begin{proof}[Proof of \cref{thm:superelliptic}]
  We first take $(F,k) = (F_{\spc},7)$, so that $\theta = \theta_{\spc}$ as in
  the rest of this section. Let $(x,y,w)$ be a solution of
  \cref{eq:superelliptic-intro}. By \cref{prop:bruin}, $x - \theta y = \delta\omega^7$ with $\delta \in \Delta(\theta,k)$, and
  $[v] \colonequals \psi([\omega])$ is a rational point of $Z_{\theta,\delta}$
  with $\varphi_{\theta,\delta}([v]) = (x:y)$, by \cref{lem:quotient-map}. Since
  $7 \mid k$, \cref{prop:seven-adic}\eqref{item:Z-is-Klein} gives
  $\delta/f'(\theta) = r\,\nu^{3}\sigma(\nu)$ with $r \in \Q^{\times}$ and $\nu
  \in K^{\times}$, and the isomorphism $[v] \mapsto [\nu v]$ from
  $Z_{\theta,\delta}$ to $Z_0$, compatible with the maps to $\Pone$. By \cref{cor:Z0}, $\nu v \in
  \Q^{\times}\sigma^{i}(\alpha)$ for some $i \in \{0,1,2\}$, and therefore
  \[
    x - \theta y \in \Q^{\times}\cdot f'(\theta)\,\sigma^{i}\bigl(\alpha^{3}\sigma(\alpha)\bigr).
  \]
  For $\theta = \theta_{\spc}$ one computes
  \begin{equation}
    \label{eq:certificates}
    f'(\theta)\,\sigma^{i}\bigl(\alpha^{3}\sigma(\alpha)\bigr) =
    -7^{3}(1 + \theta),\quad 7^{3}(5 - 2\theta),\quad -7^{3}(4 - 3\theta)
    \qquad (i = 0,1,2).
  \end{equation}
  Hence $(x,y) = \pm(1,-1)$, $\pm(5,2)$ or $\pm(4,3)$, where $F_{\spc}$ takes
  the values $\pm 7$. This proves the theorem for $(F_{\spc},7)$.

  For $F_{\ns}$ we use \cref{lem:arithmetic-of-forms}\eqref{item:substitution}.
  Let $g = \gcd(x - 5y, 2y) \in \{1,2\}$; it equals $2$ exactly when $x$ and
  $y$ are both odd. Then $(x',y') = (x-5y,-2y)/g$ is a primitive pair with
  $g^{3}F_{\spc}(x',y') = kw^7$. Since $F_{\spc}(x',y')$ is odd by
  \cref{lem:arithmetic-of-forms}\eqref{item:v2}, and $v_2(k) \in
  \{0,3\}$, $w$ is odd, $g = 1$ if $k = 7$ and $g = 2$ if $k = 56$, and
  $F_{\spc}(x',y') = 7w^{7}$ in both cases. By the split case, $(x',y')$ is one
  of $\pm(1,-1)$, $\pm(5,2)$, $\pm(4,3)$. If $g = 1$ then $y'$ is even, so
  $(x',y') = \pm(5,2)$ and $(x,y) = \pm(0,1)$. If $g = 2$ then $y = -y'$ and $x
  = 2x' - 5y'$ with $y'$ odd, so $(x,y) = \pm(7,1)$ or $\pm(7,3)$.

  Finally, $j_{\ns}(0) = j_{\spc}(-1) = 0$, while the $j$-invariants of the
  other four points, listed in \cref{tab:candidates}, are not among the thirteen
  rational CM $j$-invariants.
\end{proof}

\begin{example}
  \label{ex:five-two}
  We follow the solution $(x,y,w) = (5,2,-1)$ of $F_{\spc}(x,y) = 7w^{7}$ through
  the proof. Here $x - \theta y = 5 - 2\theta$ has norm $-7$, so it lies in
  $\OO_K[1/7]^{\times}$, and in \cref{prop:bruin} we may take $\delta = 5 - 2\theta$ and
  $\omega = 1$; thus $[1] \in C_{\theta,\delta}(\Q)$, and $\eta = \delta/f'(\theta)$.
  With $\varpi = 1 + \theta$ and $7 = \varpi^{3}\theta^{-2}$,
  \[
    5 - 2\theta = 7 - 2\varpi = \varpi\bigl(-2 + \theta^{-2}\varpi^{2}\bigr),
  \]
  so $\kappa(5 - 2\theta) = 0 + 4 = 4 = \kappa(f'(\theta))$ and $\kappa(\eta) = 0$,
  as \cref{prop:seven-adic} predicts. Indeed, by \cref{eq:certificates} with
  $i = 1$,
  \[
    \eta = \frac{5 - 2\theta}{f'(\theta)} = 7^{-3}\,\sigma(\alpha)^{3}\sigma^{2}(\alpha)
    = r\,\nu^{3}\sigma(\nu), \qquad r = 7^{-3},\quad \nu = \sigma(\alpha).
  \]
  The point $[1] = \psi([1]) \in Z_{\theta,\delta}(\Q)$ is carried by the
  isomorphism $[v] \mapsto [\nu v]$ of \cref{lem:twisting} to $[\sigma(\alpha)]
  \in Z_0(\Q)$, which is the point $(0:1:0)$ of the Klein quartic under
  \cref{prop:klein-identity}. The other two solutions of $F_{\spc}(x,y) = \pm 7$
  correspond to $\sigma(\eta)$ and $\sigma^{2}(\eta)$, that is, to $i = 2$ and
  $i = 0$ in \cref{eq:certificates}.
\end{example}

\begin{remark}
  \label{rem:no-C}
  The curves $C_{\theta,\delta}$ are not strictly needed for this proof: by
  \cref{prop:seven-adic} one can pass directly from a solution to a rational
  point of $Z_0$. We have kept them because they are the natural objects of the
  descent, and because $\psi$ explains where $Z_{\theta,\delta}$ comes from.
  For each of the six pairs $(F,k)$, the $7$ classes $\delta \in \Delta(\theta,k)$
  with $\kappa(\delta/f'(\theta)) = 0$ give the seven coverings of $Z_0 \cong
  \calZ$ by twists of $\calF$, three of which have rational points. By
  \cref{prop:klein-twists} below, these are the only $\delta \in \Delta(\theta,k)$
  for which $Z_{\theta,\delta}$ is isomorphic to $\calZ$ over $\Q$.
\end{remark}

For $\eta \in K^{\times}$ we write $Z(\eta) \colonequals \{[v] \in
\PP(\calV_K) : \Tr(\eta\,v^{3}\sigma(v)) = 0\}$ for the plane quartic of
\cref{rem:intrinsic}, so that $Z_{\theta,\delta} = Z(\delta/f'(\theta))$ for
both roots $\theta$, and $Z_0 = Z(1)$. The following proposition, which is not
needed for the proof of \cref{thm:superelliptic}, shows that
\cref{prop:seven-adic}\eqref{item:Z-is-Klein} cannot be improved: the $7$-adic
condition is exactly the condition for $Z_{\theta,\delta}$ to be the Klein
quartic.

\begin{proposition}
  \label{prop:klein-twists}
  Let $(F,k)$ be one of the six pairs of \cref{tab:descent-data}, let $\theta$
  be the corresponding root, let $\delta \in \Delta(\theta,k)$, and let $\beta
  \colonequals \varepsilon_1\varepsilon_2^{4}$.
  \begin{enumerate}
  \item \label{item:klein-iff} The curve $Z_{\theta,\delta}$ is isomorphic to
    the Klein quartic $\calZ$ over $\Q$ if and only if
    $\kappa(\delta/f'(\theta)) = 0$. This holds for exactly $7$ of the $49$
    elements of $\Delta(\theta,k)$.
  \item \label{item:two-twists} If $\kappa(\delta/f'(\theta)) \in \{1,2,4\}$,
    then $Z_{\theta,\delta}$ is isomorphic to $Z(\beta)$ over $\Q$, and if
    $\kappa(\delta/f'(\theta)) \in \{3,5,6\}$, then it is isomorphic to
    $Z(\beta^{3})$. The three curves $\calZ$, $Z(\beta)$ and $Z(\beta^{3})$ are pairwise
    non-isomorphic over $\Q$.
  \end{enumerate}
\end{proposition}
\begin{proof}
  Put $\eta \colonequals \delta/f'(\theta)$. By
  \cref{lem:one-family}\eqref{item:bijection}, the class of $\eta$ lies in
  $\Sigma$, and $\delta \mapsto [\eta]$ is a bijection from $\Delta(\theta,k)$ to
  $\Sigma$. By \cref{lem:V}\eqref{item:V-structure}, we can write $[\eta] =
  s[\varepsilon_1\varepsilon_2^{2}] + c[\beta]$ in additive notation, with $s, c
  \in \F_7$, and then $\kappa(\eta) = 4c$, since $\kappa(\varepsilon_1\varepsilon_2^{2})
  = 0$ and $\kappa(\beta) = 4$. In particular $\kappa(\eta) = 0$ for exactly $7$
  of the $49$ classes. Since $[\varepsilon_1\varepsilon_2^{2}]$ lies in
  $\Sigma_2 = (3+\sigma)\Sigma$, the argument of the proof of
  \cref{prop:seven-adic}\eqref{item:in-image} gives $\eta = r\,\nu^{3}\sigma(\nu)\,\beta^{c}$ with $r
  \in \Q^{\times}$ and $\nu \in K^{\times}$, so $Z_{\theta,\delta} =
  Z(\eta) \cong Z(\beta^{c})$ over $\Q$, by \cref{lem:twisting}. If $c = 0$, this
  is $Z(1) = Z_0 \cong \calZ$, by \cref{cor:Z0}.

  Next, for every $\eta' \in K^{\times}$ the $\Q$-linear automorphism $v
  \mapsto \sigma(v)$ of $\calV_K$ induces an isomorphism $Z(\eta') \to
  Z(\sigma(\eta'))$ over $\Q$, because
  $\Tr\bigl(\sigma(\eta')\,\sigma(v)^{3}\sigma^{2}(v)\bigr) =
  \Tr\bigl(\eta'\,v^{3}\sigma(v)\bigr)$. By \cref{lem:V}\eqref{item:V-structure},
  $\sigma[\beta] = 4[\beta]$, so $\sigma(\beta^{c})$ and $\beta^{4c}$ have the same class
  in $\Sigma$; hence $Z(\beta^{c}) \cong Z(\sigma(\beta^{c})) \cong Z(\beta^{4c})$, the
  second isomorphism by \cref{eq:seventh-powers} and \cref{lem:twisting}. Since $4$ has order $3$ in $\F_7^{\times}$, the curves
  $Z(\beta^{c})$ with $c \in \{1,4,2\}$ are isomorphic to $Z(\beta)$, and those with $c
  \in \{3,5,6\}$ to $Z(\beta^{3})$; these values of $c$ correspond to $\kappa(\eta)
  = 4c \in \{4,2,1\}$ and $\{5,6,3\}$ respectively.

  It remains to show that $Z(1)$, $Z(\beta)$ and $Z(\beta^{3})$ are pairwise
  non-isomorphic. In the coordinates $\sfa_0 + \sfa_1\theta_{\spc} +
  \sfa_2\theta_{\spc}^{2}$, the curve $Z(\eta')$ with $\eta' \in \OO_K$ is defined
  by the form $\Tr(\eta'\,\sfu_0^{3}\sfu_1) \in \Z[\sfa_0,\sfa_1,\sfa_2]$; divide it by the
  greatest common divisor of its coefficients. A computation in \Magma{} shows
  that for $\eta' \in \{1, \beta, \beta^{3}\}$ these three primitive quartic forms
  define smooth curves over $\F_{29}$, with $24$, $45$ and $17$ points over
  $\F_{29}$ respectively. So the three curves have good reduction at $29$. An
  isomorphism over $\Q_{29}$ between the generic fibers of two smooth proper
  curves of genus at least $2$ over $\Z_{29}$ extends to an isomorphism over $\Z_{29}$, because the scheme
  of isomorphisms between them is finite, hence proper, over $\Z_{29}$
  \cite{DM-69}*{Theorem 1.11}; reducing it modulo $29$ gives an isomorphism of the
  reductions, which is impossible since they have different numbers of points.
  This proves \eqref{item:two-twists}, and \eqref{item:klein-iff} follows:
  $Z_{\theta,\delta} \cong \calZ$ forces $c = 0$, that is, $\kappa(\eta) = 0$.
\end{proof}

\begin{remark}
  \label{rem:klein-twists}
  The isomorphism $Z_{\theta,\delta} \cong Z_0$ for $\kappa(\delta/f'(\theta))
  = 0$ is compatible with the maps $\varphi$ to $\Pone$, but the isomorphisms
  $Z(\beta^{c}) \cong Z(\beta^{4c})$ induced by $\sigma$ are not: they carry
  $\varphi_{\theta,\delta}$ to $[v] \mapsto [\sigma(\delta)\,v^{3}\sigma(v)]$,
  whose target is $\PP(\Q + \Q\sigma(\theta))$. So \cref{prop:klein-twists}
  classifies the quartics, not the coverings. The rational points of $Z(\beta)$ and
  $Z(\beta^{3})$ are not determined in this paper. The three curves have good
  reduction and the same number of points over $\F_p$ for every prime $p < 29$
  other than $7$, so $29$ is the smallest prime that can be used in the proof.
  (It is also the smallest prime $p \equiv 1 \pmod 7$.)
\end{remark}

\section{Proof of \texorpdfstring{\cref{thm:modular-curves}}{Theorem 2}}
\label{sec:proof-thm-2}

In this section we show that the eight non-CM points of $X_{\ns}^{+}(7)$ and $X_{\spc}^{+}(7)$ resulting from \cref{thm:superelliptic} and \cite{Furio--Lombardo26}
do not lift to rational points of $X_{\ns}^{\sharp}(49)$ or $X_{\spc}^{\sharp}(49)$.

\subsection{A Frobenius criterion}
\begin{lemma}
  \label{lem:criterion-groups}
  Let $\bullet \in \{\ns,\spc\}$, and let $g \in G^{\sharp}$
  be such that $7 \nmid \operatorname{tr}(g)\bigl(\operatorname{tr}(g)^2 -
  4\det g\bigr)$. Then $g^{48}$ is a scalar matrix.
\end{lemma}
\begin{proof}
  Let $\bar g \in C^{+}$ be the reduction of $g$. The elements of
  $C^{+}$ outside the Cartan subgroup $C$ have trace
  zero, so $\bar g \in C$; and
  $\bar g$ is not scalar, because its characteristic polynomial is separable.
  Hence the centralizer of $\bar g$ in $M_2(\F_7)$ is $\F_7[\bar g]$, a
  subalgebra of dimension $2$ of $\mathfrak{c}$, that is, $\mathfrak{c}$
  itself. Since $\bar g^{48} = 1$, we have $g^{48} = I + 7A$
  with $A \in M_2(\F_7)$; this element commutes with $g$, so $A \in
  \mathfrak{c}$, and it lies in $G^{\sharp}$, so $A \in V$. Finally,
  $V \cap \mathfrak{c}$ consists of the scalar matrices by
  \cref{lem:groups}\eqref{item:V}.
\end{proof}

\begin{corollary}
  \label{cor:criterion}
  Let $E$ be an elliptic curve over $\Q$ such that $\rho_{E,49}(\Gal_\Q)$ is
  conjugate to a subgroup of $G_{\ns}^{\sharp}(49)$ or of
  $G_{\spc}^{\sharp}(49)$. Then for every prime $p \neq 7$ of good reduction such
  that $7 \nmid a_p(E)\bigl(a_p(E)^2 - 4p\bigr)$, the matrix
  \[
    \begin{pmatrix*}[r] 0 & -p \\ 1 & a_p(E) \end{pmatrix*}^{48} \in M_2(\Z/49)
  \]
  is scalar. This condition is invariant under quadratic twists of $E$.
\end{corollary}
\begin{proof}
  The matrix $g = \rho_{E,49}(\Frob_p)$ has characteristic polynomial $\t^2 -
  a_p(E)\t + p$, which is separable modulo $7$. So $g$ has a cyclic vector
  modulo $7$, hence (by Nakayama's lemma) over $\Z/49$, and it is conjugate to
  the companion matrix of its characteristic polynomial. Now apply
  \cref{lem:criterion-groups}. A quadratic twist replaces $a_p(E)$ by $\pm
  a_p(E)$, and the companion matrix by a conjugate of its negative.
\end{proof}

The criterion is far from vacuous: in the full preimage of $C^{+}_{\ns}(7)$ in
$\GL_2(\Z/49)$, and likewise in that of $C^{+}_{\spc}(7)$, exactly $6/7$ of the
elements $g$ with $7 \nmid \operatorname{tr}(g)(\operatorname{tr}(g)^2 - 4\det g)$
violate it. For instance, an elliptic curve with the $j$-invariant of the point
$t = 5/2$ of \cref{tab:candidates} has $a_{11} = \pm 2$, and
\[
  \begin{pmatrix*}[r] 0 & -11 \\ 1 & 2 \end{pmatrix*}^{48}
  \equiv \begin{pmatrix*}[r] 43 & 42 \\ 14 & 22 \end{pmatrix*} \pmod{49}
\]
is not scalar.

\subsection{The candidates}
\Cref{tab:candidates} lists the non-CM points that occur in the proof. For each
of them we give the first prime $p$ at which the criterion of
\cref{cor:criterion} fails, together with $\lvert a_p\rvert$, for an elliptic
curve with the given $j$-invariant and good reduction at $p$. Such a curve
exists: in every row $j$ is $p$-integral and $j \not\equiv 0, 1728 \pmod p$, so
the reduction has automorphism group $\{\pm 1\}$, inertia at $p$ acts through
it, and some quadratic twist has good reduction at $p$. Its $a_p$ is determined
up to sign by $j$, and by \cref{cor:criterion} the outcome does not depend on
the sign.

\begin{table}[ht]
  \centering
  \tablesetup
  \begin{tabular}{IlIc:l:cI}
    \toprule
    \headrow
    curve & $t$ & $j$ & $(p,\lvert a_p\rvert)$ \\
    \headrule
    $X_{\ns}^{+}(7)$ & $7$ & $2^{3}\cdot 5^{3}\cdot 7^{5}$ & $(3,1)$ \\
    \rowsep
    $X_{\ns}^{+}(7)$ & $7/3$ & $2^{15}\cdot 7^{5}$ & $(3,2)$ \\
    \midrule
    $X_{\spc}^{+}(7)$ & $5/2$ & $3^{3}\cdot 5\cdot 7^{5}/2^{7}$ & $(11,2)$ \\
    \rowsep
    $X_{\spc}^{+}(7)$ & $4/3$ & $-2^{8}\cdot 5^{3}\cdot 7^{5}\cdot 37^{3}/3^{7}$ & $(13,2)$ \\
    \midrule
    $X_{\spc}^{+}(7)$ & $3/2$ & $3\cdot 5^{3}\cdot 11^{3}\cdot 17^{3}\cdot 43^{3}/2^{7}$ & $(23,4)$ \\
    \rowsep
    $X_{\spc}^{+}(7)$ & $-1/4$ & $3^{3}\cdot 37^{3}\cdot 149^{3}\cdot 2389^{3}/2^{14}$ & $(5,1)$ \\
    \rowsep
    $X_{\spc}^{+}(7)$ & $13/9$ & $-2^{18}\cdot 5^{3}\cdot 11^{3}\cdot 13\cdot 29^{3}\cdot 37^{3}\cdot 67^{3}\cdot 127^{3}/3^{14}$ & $(17,3)$ \\
    \rowsep
    $X_{\spc}^{+}(7)$ & $14/5$ & $2^{4}\cdot 3^{3}\cdot 7^{4}\cdot 19^{3}\cdot 23^{3}\cdot 43^{3}\cdot 163^{3}/5^{7}$ & $(31,10)$\\
    \bottomrule
  \end{tabular}
  \caption{Non-CM candidates, and a prime at which \cref{cor:criterion} fails.
    The first four rows come from \cref{thm:superelliptic}; the last four from
    the solutions of $F_{\spc}(x,y) = \pm 1$.}
  \label{tab:candidates}
\end{table}

\begin{proof}[Proof of \cref{thm:modular-curves}]
  Let $X$ be one of the two curves and $G$ the corresponding group. Let $P \in
  X(\Q)$ be a non-cuspidal point which is not CM. Then $j(P) \notin \{0,1728\}$,
  and since $-I \in G$ there is an elliptic curve $E$ over $\Q$ with $j(E) =
  j(P)$, without CM, such that $\rho_{E,49}(\Gal_\Q)$ is conjugate to a subgroup
  of $G$; see \cite{RSZB-22}*{Section 2.3}. Since $-I \in G$, the same holds
  for every quadratic twist of $E$; when we apply \cref{cor:criterion} at a
  prime $p$ below, we take $E$ to be a twist with good reduction at $p$, as
  explained before \cref{tab:candidates}. \Cref{prop:FL} provides coprime
  integers $x,y$ and an integer $w$ with $j(E) = j_{\bullet}(x/y)$ and $F_{\bullet}(x,y)
  = kw^{7}$.

  If $7 \mid k$, then by \cref{thm:superelliptic} the point $t = x/y$ is one of
  $0, 7, 7/3$ (non-split case) or $-1, 5/2, 4/3$ (split case). Since $j_{\ns}(0)
  = j_{\spc}(-1) = 0$, the point $t$ is one of the first four entries of
  \cref{tab:candidates}, and \cref{cor:criterion} gives a contradiction.

  If $7 \nmid k$, then $X = X_{\spc}^{\sharp}(49)$ and $k = 1$. By
  \cite{Furio--Lombardo26}*{Corollary 3.3(1), Theorem 3.5(1)} (see the proof of
  \cite{Furio--Lombardo26}*{Corollary 3.7}), $w = \pm 1$, so $(x,y)$ is a solution of
  the Thue equation $F_{\spc}(x,y) = \pm 1$. Its solutions, up to sign, are
  \begin{equation}
    \label{eq:thue-solutions}
    (1,0),\ (0,1),\ (1,1),\ (2,1),\ (3,1),\ (3,2),\ (1,-4),\ (13,9),\ (14,5),
  \end{equation}
  as one finds with the Thue solver of \Magma{}. The points $t = \infty, 0, 1,
  2, 3$ are the cusp and CM points with discriminants $-3, -19, -12, -27$.
  The remaining four are the last four entries of \cref{tab:candidates}, and
  \cref{cor:criterion} gives a contradiction again.

  It remains to show that a CM point $P \in X(\Q)$ has $j(P) = 0$. Suppose first
  that $j(P) = 1728$. Then, since $-I \in G$, some elliptic curve $E$ over $\Q$
  with $j(E) = 1728$ has $\rho_{E,49}(\Gal_\Q)$ conjugate to a subgroup of $G$
  \cite{RSZB-22}*{Section 2.3, Remark 2.1, and Section 12.2, Proposition 12.1}. Such an $E$
  is the twist of $E_1\colon y^2 = x^3 - x$ by a character $\chi$ with values
  in $\Aut(E_1) = \Z[i]^{\times}$, so that $\rho_{E,49} = \rho_{E_1,49}\chi$
  on $\Gal_{\Q(i)}$. Let $\tau \in \Gal_{\Q(i)}$ be a Frobenius element at a
  prime above $5$, which splits in $\Z[i]$. Both $\rho_{E_1,49}(\tau)$ and
  $\chi(\tau)$ lie in the Cartan subgroup $(\Z[i]/49)^{\times}$, so
  $g \colonequals \rho_{E,49}(\tau)$ is the image of $\pi\chi(\tau)$ for an
  element $\pi \in \Z[i]$ of norm $5$, and its characteristic polynomial is
  $\t^2 - a\t + 5$ with $a \in \{\pm 2, \pm 4\}$. Then $7 \nmid a(a^2 - 20)$,
  so \cref{lem:criterion-groups} says that $g^{48}$ is scalar; but $g$ is
  conjugate to the companion matrix of $\t^2 - a\t + 5$, as in the proof of
  \cref{cor:criterion}, and for each of these values of $a$ the $48$th power
  of that matrix is not scalar modulo $49$. So $j(P) \neq 1728$, and if
  $j(P) \neq 0$ then $P$ comes from an elliptic curve $E$ over $\Q$ as in the
  first paragraph of this proof,
  all of whose twists are quadratic. In the non-split case,
  \cite{Furio--Lombardo26}*{Lemma 2.11} applied to $E$ shows that $7 \mid j(P)$, and $0$
  is the only one of the thirteen rational CM $j$-invariants that is divisible
  by $7$. In the split case, write $j(P) = j_{\spc}(x/y)$ with $x,y$ coprime.
  Since $j(P)$ is an integer, \cref{lem:arithmetic-of-forms}\eqref{item:gcd}
  shows that $yF_{\spc}(x,y)$ is, up to sign, a power of $7$. If $7 \mid y$,
  then $7 \nmid xH_{\spc}(x,y)$ and $7 \nmid F_{\spc}(x,y)$, because
  $F_{\spc} \equiv \x^{3}$ and $H_{\spc} \equiv \x^{9} \pmod{\y}$, so
  $j(P) = xH_{\spc}(x,y)^{3}/(yF_{\spc}(x,y))^{7}$ is not an integer; hence
  $y = \pm 1$, and $F_{\spc}(x,y) \in \{\pm 1, \pm 7\}$ by
  \cref{lem:arithmetic-of-forms}\eqref{item:v7}. The primitive solutions of
  these Thue equations are, up to sign, the nine pairs in \cref{eq:thue-solutions} together with
  $(1,-1)$, $(5,2)$ and $(4,3)$; the CM points among these twelve are $t = 0, -1$,
  where $j = 0$, and $t = 1, 2, 3$, with discriminants $-19, -12, -27$. For $t
  \in \{1,2,3\}$ the criterion of \cref{cor:criterion} fails at
  $(p,\lvert a_p\rvert) = (5,1)$, $(13,2)$ and $(13,5)$ respectively.
\end{proof}

\begin{remark}
  \label{rem:without-2.11}
  In the non-split case, the part of the proof above that treats non-CM points
  uses \cite{Furio--Lombardo26}*{Lemma 2.11} (through \cref{prop:FL}) to force $7 \mid k$.
  This can be avoided. If $7 \nmid k$, then
  $k \in \{1,8\}$ and \cite{Furio--Lombardo26}*{Theorem 3.5(1)} gives $w = \pm 1$, so $t
  = x/y$ is one of Kenku's twelve integral points of $X_{\ns}^{+}(7)$
  \cite{Kenku-85}, namely
  \[
    t \in \{0,\ \infty,\ 1,\ -1,\ 2,\ 3,\ 5,\ -3/5,\ 7,\ 7/3,\ 11/2,\ 19/9\};
  \]
  see also
  \cite{Elkies-99}*{Section 4.3}. Eight of them are CM, two are in
  \cref{tab:candidates}, and the remaining two, $t = 11/2$ and $t = 19/9$, with
  \[
    j = 2^{6}\cdot 11^{3}\cdot 23^{3}\cdot 149^{3}\cdot 269^{3}
    \qquad\text{and}\qquad
    j = 2^{9}\cdot 17^{6}\cdot 19^{3}\cdot 29^{3}\cdot 149^{3},
  \]
  violate the criterion of \cref{cor:criterion} at $(p,\lvert a_p\rvert) = (3,1)$
  and $(3,2)$ respectively.
\end{remark}

\begin{remark}
  The first point in the split part of \cref{tab:candidates}, $j = 3^3\cdot
  5\cdot 7^5/2^7$ at $t = 5/2$, is the exceptional $j$-invariant of level $7$ of
  \cite{Zywina-15}*{Theorem 1.5(iii)}, the constant function $J_1$ there. The proof
  of \cite{Furio--Lombardo26}*{Corollary 3.7} excludes the candidates with Zywina's algorithm
  \cite{Zywina-22} instead of \cref{cor:criterion}.
\end{remark}

\section{Origin of the ideas and use of AI}
\label{sec:ai}

This section records how the ideas in this paper arose, and what AI models
contributed to it (see the AI use disclosure in the introduction). All of the
work has been checked by the authors.

Our starting point was the work of Furio and Lombardo
\cite{Furio--Lombardo26}, who, through a careful analysis of the ramification
of the $j$-maps, reduced the difficult problem of determining the rational
points on the genus $9$ curves $X_{\ns}^{\sharp}(49)$ and
$X_{\spc}^{\sharp}(49)$ to that of solving the superelliptic equations
$F(x,y) = kw^7$ in integers.

Our approach was inspired by Kummer's proof of Fermat's Last Theorem for
regular primes \cite{Kummer-1850}, as presented by Conrad \cite{Conrad-FLT}.
It led us to the curves $C_{\theta,\delta}$, and to rediscover the work of
Bruin \cites{Bruin02-thesis,Bruin06-nn2}, who uses the same approach to
parametrize the solutions of certain spherical generalized Fermat equations. We
first worked with the curves $C_{\theta,\delta}$ directly over $\Q$. We found a
map from each of them to a twist of the Klein quartic by computing their zeta
functions at several primes and matching the resulting $L$-factors with those
of curves of genus $3$, using the LMFDB \cite{LMFDB} among other resources.

We then gave our work to Claude Opus 5, which pointed out that over $K$ the
curves $C_{\theta,\delta}$ are twists of the Fermat septic
(\cref{lem:fermat-twist}), using the intrinsic description
$\PP^2(\Q) = K^{\times}/\Q^{\times}$ of \cref{rem:intrinsic}. This is not
logically necessary for the proof, but we found it a valuable conceptual
clarification, and rewrote part of the paper from this perspective.

At this stage we still had to determine the rational points on $49$ curves of
genus $15$. We knew that they admitted quotients of genus $3$, and observed that those quotients were twists of the Klein quartic. An ad hoc local argument at $7$ reduced the number of twists to around 7. From here, it was clear to us that the approach of Furio--Lombardo would lead to success, and we were in the process of working out those details.

Meanwhile, after a talk by the second author on our partial progress at the
ChaBONNty conference in Bonn,\footnote{Max Planck Institute for Mathematics,
  June 29 to July 3, 2026; \url{https://www.mpim-bonn.mpg.de/ChaBONNty}.} Filip
Najman started exploring this problem independently. He kindly informed us that
he was able to determine the integer solutions of the remaining superelliptic
equations, and that everything reduced to the Klein quartic. This surprised us,
since our computations suggested that some nontrivial twists remained.

Finally, we gave our partial work to Claude Fable 5.1, which produced the
following simplifications.
\begin{itemize}
\item Since $\theta_{\ns} = 5 - 2\theta_{\spc}$, the two roots span, together
  with $1$, the same plane $\Q + \Q\theta$ in $K$. So the two descents involve
  the same curves, and only the coordinate on the target changes
  (\cref{rem:intrinsic}, \cref{lem:one-family}\eqref{item:same-curves},
  \cref{lem:arithmetic-of-forms}\eqref{item:substitution}). This is not
  necessary for the result, but it saved us from duplicating our work.
\item Our original version of the $7$-adic step (\cref{ss:7-adic}) used an ad
  hoc local test that ruled out most of the classes $\delta$; based on our code Claude suggested \cref{prop:seven-adic} as a cleaner formulation of the method.
\item We had misread \cite{Furio--Lombardo26}*{Corollary 3.3}, and were
  duplicating some of their work with our own methods; Claude pointed out that
  several of the cases we were considering were already settled by that
  corollary.
\end{itemize}
Putting everything together, Najman turned out to be right: every twist that
survives the $7$-adic step is isomorphic over $\Q$ to the Klein quartic
(\cref{prop:seven-adic}\eqref{item:Z-is-Klein}), and our results follow from
there.


\newpage
\begin{bibdiv}
\begin{biblist}

\bib{BBHJM-25}{misc}{
      author={Balakrishnan, Jennifer~S.},
      author={Betts, L.~Alexander},
      author={Hast, Daniel~Rayor},
      author={Jha, Aashraya},
      author={M\"{u}ller, J.~Steffen},
       title={Rational points on the non-split {C}artan modular curve of level
  27 and quadratic {C}habauty over number fields},
        date={2025},
        note={Preprint,
  \href{https://arxiv.org/abs/2501.07833v1}{\texttt{arXiv:2501.07833v1}},
  January 14, 2025},
}

\bib{Magma}{article}{
      author={Bosma, Wieb},
      author={Cannon, John},
      author={Playoust, Catherine},
       title={The {M}agma algebra system. {I}. {T}he user language},
        date={1997},
     journal={J. Symbolic Comput.},
      volume={24},
      number={3-4},
       pages={235\ndash 265},
}

\bib{BD-13}{article}{
      author={Bennett, Michael~A.},
      author={Dahmen, Sander~R.},
       title={Klein forms and the generalized superelliptic equation},
        date={2013},
     journal={Ann. of Math. (2)},
      volume={177},
      number={1},
       pages={171\ndash 239},
}

\bib{Bruin02-thesis}{book}{
      author={Bruin, N.~R.},
       title={Chabauty methods and covering techniques applied to generalized
  {F}ermat equations},
      series={CWI Tract},
   publisher={Stichting Mathematisch Centrum, Centrum voor Wiskunde en
  Informatica, Amsterdam},
        date={2002},
      volume={133},
        ISBN={90-6196-508-X},
        note={Dissertation, University of Leiden, Leiden, 1999},
      review={\MR{1916903}},
}

\bib{Bruin06-nn2}{incollection}{
      author={Bruin, Nils},
       title={Some ternary {D}iophantine equations of signature {$(n,n,2)$}},
        date={2006},
   booktitle={Discovering mathematics with {M}agma},
      series={Algorithms Comput. Math.},
      volume={19},
   publisher={Springer, Berlin},
       pages={63\ndash 91},
         url={https://doi.org/10.1007/978-3-540-37634-7_3},
      review={\MR{2278923}},
}

\bib{Conrad-FLT}{misc}{
      author={Conrad, Keith},
       title={Fermat's {L}ast {T}heorem for regular primes},
        note={Expository note,
  \url{https://kconrad.math.uconn.edu/blurbs/gradnumthy/fltreg.pdf}, accessed
  October 2, 2026},
}

\bib{DG-95}{article}{
      author={Darmon, Henri},
      author={Granville, Andrew},
       title={On the equations {$z^m=F(x,y)$} and {$Ax^p+By^q=Cz^r$}},
        date={1995},
     journal={Bull. London Math. Soc.},
      volume={27},
      number={6},
       pages={513\ndash 543},
}

\bib{DM-69}{article}{
      author={Deligne, P.},
      author={Mumford, D.},
       title={The irreducibility of the space of curves of given genus},
        date={1969},
     journal={Inst. Hautes \'Etudes Sci. Publ. Math.},
      number={36},
       pages={75\ndash 109},
}

\bib{DR-73}{incollection}{
      author={Deligne, P.},
      author={Rapoport, M.},
       title={Les sch\'emas de modules de courbes elliptiques},
        date={1973},
   booktitle={Modular functions of one variable, {II} ({P}roc. {I}nternat.
  {S}ummer {S}chool, {U}niv. {A}ntwerp, {A}ntwerp, 1972)},
      series={Lecture Notes in Math.},
      volume={349},
   publisher={Springer, Berlin-New York},
       pages={143\ndash 316},
}

\bib{Elkies-99}{incollection}{
      author={Elkies, Noam~D.},
       title={The {K}lein quartic in number theory},
        date={1999},
   booktitle={The eightfold way},
      series={Math. Sci. Res. Inst. Publ.},
      volume={35},
   publisher={Cambridge Univ. Press, Cambridge},
       pages={51\ndash 101},
}

\bib{Faddeev-61}{article}{
      author={Faddeev, D.~K.},
       title={The group of divisor classes on some algebraic curves},
        date={1961},
     journal={Soviet Math. Dokl.},
      volume={2},
       pages={67\ndash 69},
        note={Translated from Dokl. Akad. Nauk SSSR \textbf{136} (1961),
  296--298},
}

\bib{Furio--Lombardo26}{article}{
      author={Furio, Lorenzo},
      author={Lombardo, Davide},
       title={On {$7$}-adic {G}alois representations for elliptic curves over
  {$\mathbf{Q}$}},
        date={2026},
        ISSN={0024-6115,1460-244X},
     journal={Proc. Lond. Math. Soc. (3)},
      volume={133},
      number={2},
       pages={Paper No. e70193, 41},
         url={https://doi.org/10.1112/plms.70193},
      review={\MR{5115964}},
}

\bib{FLS-18}{article}{
      author={Fit\'e, Francesc},
      author={Lorenzo~Garc\'ia, Elisa},
      author={Sutherland, Andrew~V.},
       title={Sato-{T}ate distributions of twists of the {F}ermat and the
  {K}lein quartics},
        date={2018},
     journal={Res. Math. Sci.},
      volume={5},
      number={4},
       pages={Paper No. 41, 40},
}

\bib{Furio-24}{misc}{
      author={Furio, Lorenzo},
       title={Effective bounds for adelic {G}alois representations attached to
  elliptic curves over the rationals},
        date={2024},
        note={Preprint,
  \href{https://arxiv.org/abs/2412.10340v4}{\texttt{arXiv:2412.10340v4}}, June
  5, 2026},
}

\bib{Genocchi-64}{article}{
      author={Genocchi, Angelo},
       title={Intorno all'equazione {$x^7+y^7+z^7=0$}},
     journal={Annali di Matematica Pura ed Applicata (1850--1865)},
      volume={6},
      number={1},
       pages={287--288},
        date={1864},
         doi={10.1007/BF03198884},
}

\bib{GR-78}{article}{
      author={Gross, Benedict~H.},
      author={Rohrlich, David~E.},
       title={Some results on the {M}ordell-{W}eil group of the {J}acobian of
  the {F}ermat curve},
        date={1978},
     journal={Invent. Math.},
      volume={44},
      number={3},
       pages={201\ndash 224},
}

\bib{HK-03}{article}{
      author={Halberstadt, Emmanuel},
      author={Kraus, Alain},
       title={Sur la courbe modulaire {$X_E(7)$}},
        date={2003},
     journal={Experiment. Math.},
      volume={12},
      number={1},
       pages={27\ndash 40},
}

\bib{Hurwitz-08}{article}{
      author={Hurwitz, A.},
       title={\"{U}ber die diophantische {G}leichung {$x^3y+y^3z+z^3x=0$}},
        date={1908},
     journal={Math. Ann.},
      volume={65},
      number={3},
       pages={428\ndash 430},
}

\bib{Karabiyik-26}{misc}{
      author={Karabiyik, Eray},
       title={Rational points on a twist of the {K}lein quartic and the
  classification of {$7$}-adic images of {G}alois representations for elliptic
  curves over {$\mathbb{Q}$}},
        date={2026},
        note={Private communication, September 18, 2026; abstract available at
  \url{https://eekarabiyik.github.io/files/7adic_abstract.pdf}, accessed
  October 2, 2026},
}

\bib{Kenku-85}{article}{
      author={Kenku, M.~A.},
       title={A note on the integral points of a modular curve of level {$7$}},
        date={1985},
     journal={Mathematika},
      volume={32},
      number={1},
       pages={45\ndash 48},
}

\bib{Klein-1879}{article}{
      author={Klein, Felix},
       title={Ueber die {T}ransformation siebenter {O}rdnung der elliptischen
  {F}unctionen},
        date={1878},
     journal={Math. Ann.},
      volume={14},
      number={3},
       pages={428\ndash 471},
        note={English translation in \emph{The eightfold way}, Math. Sci. Res.
  Inst. Publ. 35, Cambridge Univ. Press, 1999, pp. 287--331},
}

\bib{KM-85}{book}{
      author={Katz, Nicholas~M.},
      author={Mazur, Barry},
       title={Arithmetic moduli of elliptic curves},
      series={Annals of Mathematics Studies},
   publisher={Princeton University Press, Princeton, NJ},
        date={1985},
      volume={108},
}

\bib{Kummer-1850}{article}{
      author={Kummer, E.~E.},
       title={Allgemeiner {B}eweis des {F}ermat'schen {S}atzes, dass die
  {G}leichung {$x^\lambda + y^\lambda = z^\lambda$} durch ganze {Z}ahlen
  unl\"osbar ist, f\"ur alle diejenigen {P}otenz-{E}xponenten {$\lambda$},
  welche ungerade {P}rimzahlen sind und in den {Z}\"ahlern der ersten
  {$\frac{1}{2}(\lambda - 3)$} {B}ernoulli'schen {Z}ahlen als {F}actoren nicht
  vorkommen},
        date={1850},
     journal={J. Reine Angew. Math.},
      volume={40},
       pages={130\ndash 138},
}

\bib{Lame-1840}{article}{
      author={Lam\'e, G.},
       title={M\'emoire d'analyse ind\'etermin\'ee, d\'emontrant que
  l'\'equation {$x^7 + y^7 = z^7$} est impossible en nombres entiers},
        date={1840},
     journal={J. Math. Pures Appl.},
      volume={5},
       pages={195\ndash 210},
}

\bib{LGV-24}{article}{
      author={Lorenzo~Garc\'ia, Elisa},
      author={Vullers, Michael},
       title={Counterexamples to the {H}asse principle among the twists of the
  {K}lein quartic},
        date={2024},
     journal={Indag. Math. (N.S.)},
      volume={35},
      number={4},
       pages={638\ndash 645},
}

\bib{LMFDB}{misc}{
      author={{{LMFDB}{}{}~Collaboration}, The},
       title={The {L}-functions and modular forms database},
        date={2026},
        note={\url{https://www.lmfdb.org}},
}

\bib{Mazur-77}{incollection}{
      author={Mazur, B.},
       title={Rational points on modular curves},
        date={1977},
   booktitle={Modular functions of one variable, {V} ({P}roc. {S}econd
  {I}nternat. {C}onf., {U}niv. {B}onn, {B}onn, 1976)},
      series={Lecture Notes in Math.},
      volume={601},
   publisher={Springer, Berlin-New York},
       pages={107\ndash 148},
}

\bib{McCallum-88}{article}{
      author={McCallum, William~G.},
       title={On the {S}hafarevich-{T}ate group of the {J}acobian of a quotient
  of the {F}ermat curve},
        date={1988},
     journal={Invent. Math.},
      volume={93},
      number={3},
       pages={637\ndash 666},
}

\bib{MT-03}{incollection}{
      author={McCallum, William~G.},
      author={Tzermias, Pavlos},
       title={On {S}hafarevich-{T}ate groups and the arithmetic of {F}ermat
  curves},
        date={2003},
   booktitle={Number theory and algebraic geometry},
      series={London Math. Soc. Lecture Note Ser.},
      volume={303},
   publisher={Cambridge Univ. Press, Cambridge},
       pages={203\ndash 226},
}

\bib{Neukirch-99}{book}{
      author={Neukirch, J\"{u}rgen},
       title={Algebraic number theory},
      series={Grundlehren der mathematischen Wissenschaften},
   publisher={Springer-Verlag, Berlin},
        date={1999},
      volume={322},
}

\bib{PS-97}{article}{
      author={Poonen, Bjorn},
      author={Schaefer, Edward~F.},
       title={Explicit descent for {J}acobians of cyclic covers of the
  projective line},
        date={1997},
     journal={J. Reine Angew. Math.},
      volume={488},
       pages={141\ndash 188},
}

\bib{PSS-07}{article}{
      author={Poonen, Bjorn},
      author={Schaefer, Edward~F.},
      author={Stoll, Michael},
       title={Twists of {$X(7)$} and primitive solutions to {$x^2+y^3=z^7$}},
        date={2007},
     journal={Duke Math. J.},
      volume={137},
      number={1},
       pages={103\ndash 158},
}

\bib{RSZB-22}{article}{
      author={Rouse, Jeremy},
      author={Sutherland, Andrew~V.},
      author={Zureick-Brown, David},
       title={{$\ell$}-adic images of {G}alois for elliptic curves over
  {$\mathbb{Q}$} (and an appendix with {J}ohn {V}oight)},
        date={2022},
        ISSN={2050-5094},
     journal={Forum Math. Sigma},
      volume={10},
       pages={Paper No. e62, 63},
         url={https://doi.org/10.1017/fms.2022.38},
        note={With an appendix with John Voight},
      review={\MR{4468989}},
}

\bib{Serre-72}{article}{
      author={Serre, Jean-Pierre},
       title={Propri\'et\'es galoisiennes des points d'ordre fini des courbes
  elliptiques},
        date={1972},
     journal={Invent. Math.},
      volume={15},
      number={4},
       pages={259\ndash 331},
}

\bib{Silverman-94}{book}{
      author={Silverman, Joseph~H.},
       title={Advanced topics in the arithmetic of elliptic curves},
      series={Graduate Texts in Mathematics},
   publisher={Springer-Verlag, New York},
        date={1994},
      volume={151},
}

\bib{SZ-17}{article}{
      author={Sutherland, Andrew~V.},
      author={Zywina, David},
       title={Modular curves of prime-power level with infinitely many rational
  points},
        date={2017},
     journal={Algebra Number Theory},
      volume={11},
      number={5},
       pages={1199\ndash 1229},
}

\bib{Xuan2026}{misc}{
      author={Xuan, Tho~Nguyen},
       title={Complete classification of {$7$}-adic {G}alois images for
  non-{CM} elliptic curves over {$\mathbb{Q}$}},
        date={2026},
        note={Preprint,
  \href{https://arxiv.org/abs/2609.35460v1}{\texttt{arXiv:2609.35460v1}},
  September 28, 2026},
}

\bib{Zywina-15}{misc}{
      author={Zywina, David},
       title={On the possible images of the mod {$\ell$} representations
  associated to elliptic curves over {$\mathbb{Q}$}},
        date={2015},
        note={Preprint,
  \href{https://arxiv.org/abs/1508.07660v1}{\texttt{arXiv:1508.07660v1}},
  August 31, 2015},
}

\bib{Zywina-22}{misc}{
      author={Zywina, David},
       title={Explicit open images for elliptic curves over {$\mathbb{Q}$}},
        date={2022},
        note={Preprint,
  \href{https://arxiv.org/abs/2206.14959v2}{\texttt{arXiv:2206.14959v2}}, March
  22, 2024},
}

\end{biblist}
\end{bibdiv}
\end{document}